\documentclass[11pt]{article}
\usepackage[inner=3.4 cm,outer=3.4 cm,top=3.0 cm,bottom=2.0 cm]{geometry}

\usepackage{amsmath,amssymb,amsfonts,amsthm}
\usepackage{setspace}
\newcounter{item}[section]
\newcounter{kirshr}
\newcounter{kirsha}
\newcounter{kirshb}
\newenvironment{enumroman}{\setcounter{kirshr}{1}
\begin{list}{(\roman{kirshr})}{\usecounter{kirshr}} }{\end{list}}
\newtheorem{theorem}{Theorem}[section]

\newtheorem{lemma}[theorem]{Lemma}
\newtheorem{corollary}[theorem]{Corollary}
\theoremstyle{definition}

\newtheorem{example}[theorem]{Example}
\newtheorem{definition}[theorem]{Definition}
\newtheorem{proposition}[theorem]{Proposition}
\def\R{\mathbb{R}}
\def\Q{\mathbb{Q}}

\def\A{{\mathfrak{A}}}
\def\At{{\sf At}}
\def\B{{\mathfrak{B}}}

\def\C{{\mathfrak{C}}}
\def\Ca{{\mathfrak Ca}}
\def\CA{{\sf CA}}
\def\Cm{{\sf Cm}}
\def\CRCA{{\sf CRCA}}
\def\de{Dedekind-MacNeille}
\def\ef{Ehren\-feucht--Fra\"\i ss\'e}
\def\F{{\mathfrak{F}}}

\def\g{{\sf g}}
\def\Id{{\sf Id}}
\def\K{{\sf K}}

\def\Mo{{\sf M}}
\def\N{{\mathbb{N}}}
\def\nodes{{\sf nodes}}
\def\Nr{{\mathfrak{Nr}}}
\def\Nrr{{\mathfrak{Nr}}}
\def\pa{$\forall$}

\def\pe{$\exists$}
\def\PEA{\sf PEA}
\def\QEA{{\sf QEA}}
\def\r{{\sf r}}
\def\R{\mathfrak{R}}
\def\Ra{{\sf Ra}}
\def\RCA{{\sf RCA}}
\def\Rd{{\mathfrak{Rd}}}
\def\Rl{{\mathfrak Rl}}
\def\rng{{\sf rng}}

\def\Sc{{\sf Sc}}
\def\Sg{{\mathfrak Sg}}
\def\Tm{{\mathfrak{Tm}}}

\def\w{{\sf w}}
\def\ws{winning strategy}
\def\y{{\sf y}}
\def\Z{{\mathbb{Z}}}
\def\D{{\mathfrak{D}}}

\def\Ra{{\sf Ra}}
\def\set#1{ \{#1\}}
\def\Mo{{\sf Mo}}
\def\restr #1{{\restriction_{#1}}}
\def\ws{winning strategy}

 \def\CA{{\sf CA}}
\def\B{{\sf B}}

\def\w{{\sf w}}
\def\y{{\sf y}}
\def\g{{\sf g}}
\def\b{{\sf b}}
\def\r{{\sf r}}
\def\K{{\sf K}}

\def\pa{$\forall$}
\def\pe{$\exists$}

\def\ef{Ehren\-feucht--Fra\"\i ss\'e}

\def\M{{\mathfrak{M}}}

\def\Ca{{\mathfrak{Ca}}}

\def\M{{\mathfrak{M}}}

\def\A{{\mathfrak{A}}}
\def\B{{\mathfrak{B}}}
\def\C{{\mathfrak{C}}}
\def\D{{\mathfrak{D}}}

\def\M{{\mathfrak{M}}}
\def\dom{{\sf dom}}
\def\rng{{\sf rng}}
\def\Rd{{\sf{Rd}}}
\def\At{{\sf At}}
\def\Ra{{\mathfrak{Ra}}}
\def\Tm{{\mathfrak{Tm}}}
\def\Cm{{\mathfrak{Cm}}}
\def\E{{\mathfrak{E}}}
\def\Bl{{\mathfrak{Bl}}}

\def\ef{Ehren\-feucht--Fra\"\i ss\'e}

\def\Id{{\sf Id}}
\def\Rl{{\mathfrak{Rl}}}
\def\F{{\mathfrak{F}}}
\def\Sc{{\mathfrak{Sc}}}
\def\QEA{{\sf QEA}}

\def\RCA{{\sf RCA}}
\def\QEA{{\bf PEA}}

\def\R{{\sf R}}

\def\U{\mathfrak{U}}

\def\Df{{\sf Df}}

\def\Rd{{\mathfrak{Rd}}}
\def\PEA{{\sf QEA}}
\def\la#1{\langle#1\rangle}

\def\Sc{{\sf Sc}}
\def\nodes{{\sf nodes}}

\def\QEA{{\sf PEA}}

\def\cyl#1{{\sf c}_{#1}}

\def\cyl#1{{\sf c}_{#1}}
\def\sub#1#2{{\sf s}^{#1}_{#2}}
\def\diag#1#2{{\sf d}_{#1#2}}

\def\V{{\sf V}}
\def\de{Dedekind-MacNeille}

\def\QA{{\sf QA}}

\def\Z{{\mathbb{Z}}}

\def\Cs{{\sf Cs}}
\def\Ra{{\sf Ra}}
\def\Mo{{\sf M}}
\def\QEA{{\sf QEA}}

\def\RA{{\sf RA}}
\def\CRCA{{\sf CRCA}}

\def\b#1{{\bar{ #1}}}

\def\Nrr{\mathfrak{{Nr}}}
\def\w{{\sf w}}
\def\g{{\sf g}}
\def\y{{\sf y}}
\def\r{{\sf r}}

\def\bb{{\sf b}} 

\def\PEA{{\sf PEA}}
\def\R{\mathfrak{R}}
\def\Nr{{\sf Nr}}
\def\Q{{\mathbb{Q}}}
\def\H{{\cal{H}}}
\def\P{{\cal P}}

\makeindex             

\title{Non elementary classes of relation and cylindric algebras}%
\author{Tarek Sayed Ahmed\\
Department of Mathematics, Faculty of Science,\\
Cairo University, Giza, Egypt.}
\date{}
\begin{document}
\maketitle

\begin{abstract} For any pair of ordinals $\alpha<\beta$, $\sf CA_\alpha$ 
denotes the class of cylindric algebras of dimension $\alpha$, $\sf RCA_{\alpha}$ denote te class of represntable $\CA_{\alpha}$s 
and $\Nr_{\alpha}\CA_{\beta}$ ($\Ra\CA_{\beta})$ denotes the class of $\alpha$-neat reducts (relation algebra reducts) of $\CA_{\beta}$.
We show that any class $\sf K$ such that $\Ra\CA_{\omega}\subseteq \sf K\subseteq  \Ra\CA_5$, $\sf K$  is not elementary, i.e not definable in first order logic. Let $2<n<\omega$.
It is also shown that any class $\sf K$ such that $\Nr_n\CA_{\omega}\cap {\sf CRCA}_n\subseteq {\sf K}\subseteq \bold S_c\Nr_n\CA_{n+3}$, where $\CRCA_n$ is the class of completely representable
$\CA_n$s, and $\bold S_c$ denotes the operation of forming complete subalgebras, is proved not to be elementary. Finally, we show that any class $\sf K$ such that $\bold S_d\Ra\CA_{\omega}
\subseteq {\sf K}\subseteq \bold S_c\Ra\CA_5$ is not elementary. It remains to be seen whether  there exist elementary classes between $\Ra\CA_{\omega}$ 
and $\bold S_d\R\CA_{\omega}$.
In particular, for $m\geq n+3$, the classes $\Nr_n\CA_m$, $\CRCA_n$, $\bold S_d\Nr_n\CA_m$, where $\bold S_d$ is the operation of forming dense subalgebras  
are not first order definable
\end{abstract}
\section{Introduction}
We follow the notation of \cite{1} which is in conformity with the notation in the monograph. Relation algebras $\sf RA$s are abstractions of algebras whose universe consists of binary relations,  
together with the Boolean operations of intersection, union and complemention, and additional binary operation of compoition of relations, the unary one forming converses of relations and the identity relation as a constant element in the signature. 
We consider relation algebras as algebras of the form ${\cal R}=\langle R, +, \cdot, -, 1', \smile, ; , \rangle$, where $\langle R, +, \cdot , -\rangle$ is a Boolean algebra $1'\in R$, $\smile$ is the unary operation of forming conversesand $;$ is th binary operation of compotion. 
A relation algebra is {\it representable}$\iff$ it is isomorphic to a subalgebra of the form $\langle \wp(X), \cup, \cap, \sim, \smile, \circ, Id\rangle$, where $X$ is an equivalence relation, $1'$ is interpreted as the identity relation, $\smile$ is the operation of forming converses, 
and$;$  is interpreted as composition of relations. 
Following standard notation, $(\sf R)RA$ denotes the class of (representable) relation algebras. The class $\sf RA$ is  a discriminator variety that is finitely axiomatizable, cf. \cite[Definition 3.8, Theorems 3.19]{HHbook}.  
Algebras consisting of relations of possibly higher arity $\alpha$, $\alpha$ an arbitary ordinal are represented by cylindric algebras of dimension $\beta$. 
In this case the concrete versions of such algebras are Boolean algbras with operators where the extra Booolean operations are the projections or cylindrifcations ${\sf C}_k$ ($k<\beta)$ 
and so-called diagonal elements reflecting equality. 
For a set $V$, ${\cal B}(V)$ denotes the Boolean set algebra $\langle \wp(V), \cup, \cap, \sim, \emptyset, V\rangle$.
Let $U$ be a set and $\alpha$ an ordinal; $\alpha$ will be the dimension of the algebra.
For $s,t\in {}^{\alpha}U$ write $s\equiv_i t$ if $s(j)=t(j)$ for all $j\neq i$.
For $X\subseteq {}^{\alpha}U$ and $i,j<\alpha,$ let
$${\sf C}_iX=\{s\in {}^{\alpha}U: (\exists t\in X) (t\equiv_i s)\}$$
and
$${\sf D}_{ij}=\{s\in {}^{\alpha}U: s_i=s_j\}.$$
$\langle{\cal B}(^{\alpha}U), {\sf C}_i, {\sf D}_{ij}\rangle_{i,j<\alpha}$ is called {\it the full cylindric set algebra of dimension $\alpha$}
with unit (or greatest element) $^{\alpha}U$. Any subalgebra of the latter is called a {\it set algebra of dimension $\alpha$}.
Examples of subalgebras of such set algebras arise naturally from models of first order theories.
Indeed, if $\Mo$ is a first order structure in a first
order signature $L$ with $\alpha$ many variables, then one manufactures a cylindric set algebra based on $\Mo$ as follows, cf. \cite[\S4.3]{HMT2}.
Let
$$\phi^{\Mo}=\{ s\in {}^{\alpha}{\Mo}: \Mo\models \phi[s]\},$$
(here $\Mo\models \phi[s]$ means that $s$ satisfies $\phi$ in $\Mo$), then the set
$\{\phi^{\Mo}: \phi \in Fm^L\}$ is a cylindric set algebra of dimension $\alpha$, where $Fm^L$ denotes the set of first order formulas taken in 
the signature $L$.  To see why, we have:
\begin{align*}
\phi^{\Mo}\cap \psi^{\Mo}&=(\phi\land \psi)^{\Mo},\\
^{\alpha}{\Mo}\sim \phi^{\Mo}&=(\neg \phi)^{\Mo},\\
{\sf C}_i(\phi^{\Mo})&=(\exists v_i\phi)^{\Mo},\\
{\sf D}_{ij}&=(x_i=x_j)^{\Mo}.
\end{align*}
Following \cite{HMT2}, $\Cs_{\alpha}$ denotes the class of all subalgebras of full set algebras of dimension $\alpha$.
The (equationally defined) $\CA_{\alpha}$ class is obtained from cylindric set algebras by a process of abstraction and is defined by a {\it finite} schema
of equations given in \cite[Definition 1.1.1]{HMT2} that holds of course in the more concrete set algebras. 
\begin{definition} Let $\alpha$ be an ordinal. By \textit{a cylindric algebra of dimension} $\alpha$, briefly a
$\CA_{\alpha}$, we mean an
algebra
$$ {\A} = \langle A, +, \cdot,-, 0 , 1 , {\sf c}_i, {\sf d}_{ij}\rangle_{\kappa, \lambda < \alpha}$$ where $\langle A, +, \cdot, -, 0, 1\rangle$
is a Boolean algebra such that $0, 1$, and ${\sf d}_{i j}$ are
distinguished elements of $A$ (for all $j,i < \alpha$),
$-$ and ${\sf c}_i$ are unary operations on $A$ (for all
$i < \alpha$), $+$ and $.$ are binary operations on $A$, and
such that the following equations  are satisfied for any $x, y \in
A$ and any $i, j, \mu < \alpha$:
\begin{enumerate}
\item [$(C_1)$] $  {\sf c}_i 0 = 0$,
\item [$(C_2)$]$  x \leq {\sf c}_i x \,\ ( i.e., x + {\sf c}_i x = {\sf c}_i x)$,
\item [$(C_3)$]$  {\sf c}_i (x\cdot {\sf c}_i y )  = {\sf c}_i x\cdot  {\sf c}_i y $,
\item [$(C_4)$] $  {\sf c}_i {\sf c}_j x   = {\sf c}_j {\sf c}_i x $,
\item [$(C_5)$]$  {\sf d}_{i i} = 1 $,
\item [$(C_6)$]if $  i \neq j, \mu$, then
 ${\sf d}_{j \mu} = {\sf c}_i
 ( {\sf d}_{j i} \cdot  {\sf d}_{i \mu}  )  $,
\item [$(C_7)$] if $  i \neq j$, then
 ${\sf c}_i ( {\sf d}_{i j} \cdot  x) \cdot   {\sf c}_i
 ( {\sf d}_{i j} \cdot  - x) = 0
 $.
\end{enumerate}
\end{definition}Let $\alpha$ be an ordinal and $\A\in \CA_{\alpha}$. For any $i, j, l<\alpha$, let ${\sf s}_i^jx=x$ if $i=j$ and ${\sf s}_i^jx={\sf c}_j({\sf d}_{ij}\cdot x)$ 
if $i\neq j$. Let $_l{\sf s}(i, j)x={\sf s}^l_i{\sf s}_j^i{\sf s}_l^jx$.
In the next definition, in its first item  we define the notion of forming $\alpha$-neat reducts of $\CA_{\beta}$s with $\beta>\alpha$, in symbols $\Nrr_{\alpha}$, 
 and in the second item we define relation algebras obtained from cylindric algebras using the operator $\Nrr_2$.
\begin{definition}
\begin{enumerate} 
\item Assume that $\alpha<\beta$ are ordinals and that 
$\B\in \CA_{\beta}$. Then the {\it $\alpha$--neat reduct} of $\B$, in symbols
$\mathfrak{Nr}_{\alpha}\B$, is the
algebra obtained from $\B$, by discarding
cylindrifiers and diagonal elements whose indices are in $\beta\setminus \alpha$, and restricting the universe to
the set $Nr_{\alpha}B=\{x\in \B: \{i\in \beta: {\sf c}_ix\neq x\}\subseteq \alpha\}.$
\item Assume that $\alpha\geq 3$. Let $\A\in \CA_{\alpha}$. Then $\Ra\A=\langle Nr_2\A:+, \cdot, -, ;,  {\sf d}_{01}\rangle$ where for any $x,y\in Nr_n\A$, 
$x;y={\sf c}_2({\sf s}_2^1x\cdot {\sf s}_2^0y)$ and $x=_2{\sf s}(0.1)x$ 
\end{enumerate}
\end{definition}
If $\A\in \CA_3$, $\Ra\A$, having he same signature as $\sf RA$  may not be  a relation algebra as associativiy of the (abstract) composition operation may fail, but for $\alpha\geq 4$, $\Ra\CA_{\beta}\subseteq \RA$.
It is straightforward to check that $\mathfrak{Nr}_{\alpha}\B\in \CA_{\alpha}$. 
Let $\alpha<\beta$ be ordinals. If $\A\in \CA_\alpha$ and $\A\subseteq \mathfrak{Nr}_\alpha\B$, with $\B\in \CA_\beta$, then we say that $\A$ {\it neatly embeds} in $\B$, and 
that $\B$ is a {\it $\beta$--dilation of $\A$}, or simply a {\it dilation} of $\A$ if $\beta$ is clear 
from context. For $\bold K\subseteq \CA_{\beta}$, 
we write $\Nr_{\alpha}\bold K$ for the class $\{\mathfrak{Nr}_{\alpha}\B: \B\in \bold K\}.$
Following \cite{HMT2}, ${\sf Gs}_n$ denotes the class of {\it generalized cylindric set algebra of dimension $n$}; $\C\in {\sf Gs}_n$, if $\C$ has top element
$V$ a disjoint union of cartesian squares,  that is $V=\bigcup_{i\in I}{}^nU_i$, $I$ is a non-empty indexing set, $U_i\neq \emptyset$  
and  $U_i\cap U_j=\emptyset$  for all $i\neq j$. The operations of $\C$ are defined like in cylindric set algebras of dimension $n$ 
relativized to $V$. By the same token the variety of representable relation algebras is denoted by $\sf RRA$. It is known that $\bold I{\sf Gs}_n={\sf RCA}_n=\bold S\Nr_n\CA_{\omega}=\bigcap_{k\in \omega}\bold S\Nr_n\CA_{n+k}$
and that $\sf RRA=\bold S\Ra\CA_{\omega}=\bigcap_{k\in \omega}\bold S\Ra\CA_{3+k}$. . 
We often identify set algebras with their domain referring to an injection $f;\A\to \wp(V)$ ($\A\in \CA_n$) as a complete representation  of $\A$ (via $f$) 
where $V$  is a ${\sf Gs}_n$ unit.
\begin{definition} An algebra $\A\in {\sf CA}_n$ is {\it completely representable} $\iff$ there exists $\C\in {\sf Gs}_n$, and an isomorphism $f:\A\to \C$ such that for all $X\subseteq \A$, 
$f(\sum X)=\bigcup_{x\in X}f(x)$, whenever $\sum X$ exists in $\A$. In this case, we say that $\A$ is {\it completely representable via $f$.}
\end{definition}
It is known that $\A$ is completely representable via $f:\A\to \C$, where $\C\in {\sf Gs}_n$ has top element $V$ say 
$\iff$ $\A$ is atomic and $f$ is {\it atomic} in the sense that 
$f(\sum \At\A)=\bigcup_{x\in \At\A}f(x)=V$ \cite{HH}. We denote the class of completely representable $\CA_n$s by $\CRCA_n$.
Complete representations for $\sf RA$s are defined analogously. The cl;ass of completely representable $\sf RA$ is denoted by $\sf CRRA$. 

Unless otherwise indicated, $n$ will be a finite ordinal $>2$. In \cite{r} it is proved that any class between $\sf K$ between $\bold S_c\Ra\CA_{\omega}\cap \sf CRRA$ and $\bold S_c\Ra\CA_5$, $\sf K$ is not elementary.
Here we strenghthen this result by replacing the first $\bold S_c$  by $\bold S_d$ (the operation of forming dense subalgebras).We recall that for $\sf BAO$s $\A$ and $\B$ having the same signature, 
$\B$ is dense in $\A$, 
written $\B\subseteq_d \A$, if $\B$ is  subalgbra of $\A$ such that for all non-zero $a\in \A$, there exists  a non-zero $b\in \B$ with $b\leq a$.
It is known that for a class $\sf K$ of $\sf BAO$s $\K\subseteq \bold S_d\K\subseteq \bold S_c\K$; 
and the inclusion are proper for Boolean algebras 
(without operators).   For a class $\sf K$ of $\sf BAO$, we let $\bold S_c\K=\{\B: (\exists \A\in \K))(\forall X\subseteq \A)[\sum^{\A}X=1\implies \sum X \text{ exists  in  $\B$ and $\sum^{\B}X=1$}\}$. It is known that for any $\sf K$ of $\sf BAO$s, $\sf K\subseteq \bold S_d\K\subseteq \bold S_c\K$. 
If $\sf K$ happens to be the class of Boolen algebras (without opertors) then these incluions are proper and in the specific case we adress the inclusion is also strict thereby obtaining  a stronger result than that announced in \cite{r2}. 
More importantly, we show that any class $\sf K$, such that $\Ra\CA_{\omega}\cap {\sf CRRA}\subseteq \sf K\subseteq \Ra\CA_5$ is not elementary. In the first case, 
we use so-called Rainbow algebras, in the second case we use so-called Monk algebra where a paricular 
combinatorial form of Ramsey's theorem plays an essential role.  In \cite{IGPL} it is proved that for any pair of infinite ordinals $\alpha<\beta$, the class $\Nr_{\alpha}\CA_{\beta}$ is not elementary. 
A different model theoretic proof for finite $\alpha$ is given in \cite[Theorem 5.4.1]{Sayedneat}.
This result is extended to many cylindric like algebras like Halmos'  polyadic algbras with and without eqaulity, and Pinter's substitution algebras in \cite{Fm, note}. The class $\CRCA_n$ is 
proved not be elementary by Hirsch and Hodkinson in \cite{HH}. Neat embeddings and complete representations are linked in \cite[Theorem 5.3.6]{Sayed} where it is shown that 
$\CRCA_n$ coincides with the class $\bold S_c\Nr_n\CA_{\omega}$ on atomic 
algebra having countably many atoms. Below it is proved that this charactarization does not generalize to atomic algebras having uncountably many atoms.
Completely analogous results are obtained for $\sf RA$s , that is to say, $\bold S_c\Ra\CA_{\omega}$ and $\sf CRRA$ coincide on atomic algebras with countably many atoms, 
and this charcterization does not generalize to the case of atomic $\RA$s having uncountably many atoms.
In fact, we shall prove that there exists an atomless $\C\in \CA_{\omega}$, 
such that for all $n<\omega$, $\Nr_n\A$ and $\Ra\A$ are atomic algebras having uncountably many atoms, bu do not have a complete representation.

In this paper, we use combinatorial 
game theory combined with basic graph theory resorting to (as mentioned above) Rainbow construction which is extremely effficient and flexible in constructing subtle delicate counterexamples. 
Rainbow constructions are based on two player determinitic games and  as the name suggests they involve `colours'. 
Such games happen to be simple \ef\ forth games 
where the two players \pe lloise and \pa belard,  between them, use pebble pairs outside the board, each player pebbling one of the two structures which 
she/he sticks to it during the whole play. In the number of rounds played (that can be transfinite),  \pe\ tries 
to show that  
two simple relational structures $\sf G$ (the greens) and $\sf R$ (the reds) have similar structures  
while \pa\ tries to show that they are essentially distinct. 
Such structures may include ordered structures and complete
irreflexive graphs, such as finite ordinals, $\omega_1$, $\N$, $\Z$ or $\mathbb{R}.$    
A \ws\ for either player in the \ef\ game 
can be lifted to \ws\ in a {\it rainbow game} played on so--called atomic networks on a rainbow atom structure (for both $\CA$s and $\RA$s ) based also 
on $\sf G$ and $\sf R$. Once $\sf G$ and $\sf R$ are specified, 
the rainbow atom structure is uniquely defined. 
Though more (rainbow) colours (like whites and shades of yellow) are involved in the rainbow 
atom structure,  the crucial thing here is that the number of rounds and nodes in networks used in the rainbow game,  
depend recursively on the number of rounds and pebble pairs in the 
simple \ef\ forth two player game played on $\sf G$ and $\sf R$. 
Due to the control on \ws's in terms of the relational structures ($\sf G$ and $\sf R$) chosen in advance, 
and the number of pebble pairs used outside the board,  rainbow constructions have proved highly effective in providing subtle counterexamples to really bewildering `yes or no' assertions 
for both $\CA$s and $\RA$s (relation algebras)  cf. \cite{HH, HHbook, HHbook2, mlq}

\section{Preliminaries}

 From now on,  unless otherwise indicated, $n$ is fixed to be  a finite ordinal $>2$.
Let $i<n$. For $n$--ary sequences $\bar{x}$ and $\bar{y}$,  we write $\bar{x}\equiv_ i\bar{y}$ $\iff \bar{y}(j)=\bar{x}(j)$ for all $j\neq i$,
For $i, j<n$ the replacement $[i/j]$ is the map that is like the identity on $n$, except that $i$ is mapped to $j$ and the transposition 
$[i, j]$ is the like the identity on $n$, except that $i$ is swapped with $j$. 
\begin{definition}\label{game}  
\begin{enumerate}
\item An {\it $n$--dimensional atomic network} on an atomic algebra $\A\in \CA_n$  is a map $N: {}^n\Delta\to  \At\A$, where
$\Delta$ is a non--empty finite set of {\it nodes}, denoted by $\nodes(N)$, satisfying the following consistency conditions for all $i<j<n$: 
\begin{enumroman}
\item If $\bar{x}\in {}^n\nodes(N)$  then $N(\bar{x})\leq {\sf d}_{ij}\iff\bar{x}_i=\bar{x}_j$,
\item If $\bar{x}, \bar{y}\in {}^n\nodes(N)$, $i<n$ and $\bar{x}\equiv_i \bar{y}$, then  $N(\bar{x})\leq {\sf c}_iN(\bar{y})$,
\end{enumroman}
Let $i<n$. For $n$--dimensional atomic networks $M$ and $N$,  we write 
$M\equiv_i N\iff M(\bar{y})=N(\bar{y})$ for all $\bar{y}\in {}^{n}(n\sim \{i\})$.

\item  Assume that $\A\in \CA_n$ is  atomic and that $m, k\leq \omega$. 
The {\it atomic game $G^m_k(\At\A)$, or simply $G^m_k$}, is the game played on atomic networks
of $\A$ using $m$ nodes and having $k$ rounds \cite[Definition 3.3.2]{HHbook2}, where
\pa\ is offered only one move, namely, {\it a cylindrifier move}: \\
Suppose that we are at round $t>0$. Then \pa\ picks a previously played network $N_t$ $(\nodes(N_t)\subseteq m$), 
$i<n,$ $a\in \At\A$, $\bar{x}\in {}^n\nodes(N_t)$, such that $N_t(\bar{x})\leq {\sf c}_ia$. For her response, \pe\ has to deliver a network $M$
such that $\nodes(M)\subseteq m$,  $M\equiv _i N$, and there is $\bar{y}\in {}^n\nodes(M)$
that satisfies $\bar{y}\equiv _i \bar{x}$ and $M(\bar{y})=a$.  

We write $G_k(\At\A)$, or simply $G_k$, for $G_k^m(\At\A)$ if $m\geq \omega$.
\item The $\omega$--rounded game $\bold G^m(\At\A)$ or simply $\bold G^m$ is like the game $G_{\omega}^m(\At\A)$ 
except that \pa\ has the option 
to reuse the $m$ nodes in play.
\end{enumerate}
\end{definition}
\begin{proposition}\label{rep} Suppose that $\A\in {\sf CA}_n$ is atomic having countably many atoms. 
Then \pe\ has a \ws\ in $G_{\omega}(\At\A)\iff$ \pe\ has a \ws\ in $\bold G^{\omega}(\At\A)\iff\A\in {\sf CRCA}_n$. 
In particular, if $\A$ is finite, then 
\pe\ has a \ws\ in $G_{\omega}(\At\A)\iff \A$ 
is representable.
\end{proposition}
\begin{proof} \cite[Theorem 3.3.3]{HHbook2},  together with observing that the game $G_{\omega}(\At\A)$ is equivalent to the game $\bold G^{\omega}(\At\A)$, 
in the sense that \pe\ has a \ws\ in $\bold G^{\omega}(\At\A)\implies$ \pe\ has a \ws\ in 
$G_{\omega}(\At\A)$ whenever $\A$ is atomic with countably many atoms (the converse implication is trivial). The rest of the cases are analogous.
\end{proof} 
\begin{lemma}\label{n}
Assume that $2<n<\omega$. Let $m$ be an ordinal $>n$. If  $\A\in \bold S_c\Nr_n\CA$ is atomic, then \pe\ has a \ws\ in $\bold G^m(\At\A).$ 
\end{lemma}
\begin{proof} 
First a piece of notation. Let $m$ be a finite ordinal $>0$. An $\sf s$ word is a finite string of substitutions $({\sf s}_i^j)$ $(i, j<m)$,
a $\sf c$ word is a finite string of cylindrifications $({\sf c}_i), i<m$;
an $\sf sc$ word $w$, is a finite string of both, namely, of substitutions and cylindrifications.
An $\sf sc$ word
induces a partial map $\hat{w}:m\to m$:
$\hat{\epsilon}=Id,$ $\widehat{w_j^i}=\hat{w}\circ [i|j]$
and $\widehat{w{\sf c}_i}= \hat{w}\upharpoonright(m\smallsetminus \{i\}).$
If $\bar a\in {}^{<m-1}m$, we write ${\sf s}_{\bar a}$, or
${\sf s}_{a_0\ldots a_{k-1}}$, where $k=|\bar a|$,
for an  arbitrary chosen $\sf sc$ word $w$
such that $\hat{w}=\bar a.$
Such a $w$  exists by \cite[Definition~5.23 ~Lemma 13.29]{HHbook}.

Fix $2<n<m$. Assume that $\C\in\CA_m$, $\A\subseteq_c\mathfrak{Nr}_n\C$ is an
atomic $\CA_n$ and $N$ is an $\A$--network with $\nodes(N)\subseteq m$. Define
$N^+\in\C$ by
\[N^+ =
 \prod_{i_0,\ldots, i_{n-1}\in\nodes(N)}{\sf s}_{i_0, \ldots, i_{n-1}}{}N(i_0,\ldots, i_{n-1}).\]
For a network $N$ and  function $\theta$,  the network
$N\theta$ is the complete labelled graph with nodes
$\theta^{-1}(\nodes(N))=\set{x\in\dom(\theta):\theta(x)\in\nodes(N)}$,
and labelling defined by
$$(N\theta)(i_0,\ldots, i_{n-1}) = N(\theta(i_0), \theta(i_1), \ldots,  \theta(i_{n-1})),$$
for $i_0, \ldots, i_{n-1}\in\theta^{-1}(\nodes(N))$.
Then the following hold:

(1): for all $x\in\C\setminus\set0$ and all $i_0, \ldots, i_{n-1} < m$, there is $a\in\At\A$, such that
${\sf s}_{i_0,\ldots, i_{n-1}}a\;.\; x\neq 0$,

(2): for any $x\in\C\setminus\set0$ and any
finite set $I\subseteq m$, there is a network $N$ such that
$\nodes(N)=I$ and $x\cdot N^+\neq 0$. Furthermore, for any networks $M, N$ if
$M^+\cdot N^+\neq 0$, then
$M\restr {\nodes(M)\cap\nodes(N)}=N\restr {\nodes(M)\cap\nodes(N)},$

(3): if $\theta$ is any partial, finite map $m\to m$
and if $\nodes(N)$ is a proper subset of $m$,
then $N^+\neq 0\rightarrow {(N\theta)^+}\neq 0$. If $i\not\in\nodes(N),$ then ${\sf c}_iN^+=N^+$.


Since $\A\subseteq _c\mathfrak{Nr}_n \C$, then $\sum^{\C}\At\A=1$. 
{\it For (1), ${\sf s}^i_j$ is a
completely additive operator (any $i, j<m$), hence ${\sf s}_{i_0,\ldots, i_{n-1}}$
is, too.}
So $\sum^{\C}\set{{\sf s}_{i_0\ldots, i_{n-1}}a:a\in\At(\A)}={\sf s}_{i_0\ldots i_{n-1}}
\sum^{\C}\At\A={\sf s}_{i_0\ldots, i_{n-1}}1=1$ for any $i_0,\ldots, i_{n-1}<m$.  Let $x\in\C\setminus\set0$.  Assume for contradiction
that  ${\sf s}_{i_0\ldots, i_{n-1}}a\cdot x=0$ for all $a\in\At\A$. Then  $1-x$ will be
an upper bound for $\set{{\sf s}_{i_0\ldots i_{n-1}}a: a\in\At\A}.$
But this is impossible
because $\sum^{\C}\set{{\sf s}_{i_0\ldots, i_{n-1}}a :a\in\At\A}=1.$

To prove the first part of (2), we repeatedly use (1).
We define the edge labelling of $N$ one edge
at a time. Initially, no hyperedges are labelled.  Suppose
$E\subseteq\nodes(N)\times\nodes(N)\ldots  \times\nodes(N)$ is the set of labelled hyperedges of $
N$ (initially $E=\emptyset$) and
$x\;.\;\prod_{\bar c \in E}{\sf s}_{\bar c}N(\bar c)\neq 0$.  Pick $\bar d$ such that $\bar d\not\in E$.
Then by (1) there is $a\in\At(\A)$ such that
$x\;.\;\prod_{\bar c\in E}{\sf s}_{\bar c}N(\bar c)\;.\;{\sf s}_{\bar d}a\neq 0$.
Include the hyperedge $\bar d$ in $E$.  We keep on doing this until eventually  all hyperedges will be
labelled, so we obtain a completely labelled graph $N$ with $N^+\neq 0$.
it is easily checked that $N$ is a network.

For the second part of $(2)$, we proceed contrapositively. Assume that there is
$\bar c \in{}\nodes(M)\cap\nodes(N)$ such that $M(\bar c )\neq N(\bar c)$.
Since edges are labelled by atoms, we have $M(\bar c)\cdot N(\bar c)=0,$
so
$0={\sf s}_{\bar c}0={\sf s}_{\bar c}M(\bar c)\;.\; {\sf s}_{\bar c}N(\bar c)\geq M^+\cdot N^+$.
A piece of notation. For $i<m$, let $Id_{-i}$ be the partial map $\{(k,k): k\in m\smallsetminus\{i\}\}.$
For the first part of (3)
(cf. \cite[Lemma~13.29]{HHbook} using the notation in {\it op.cit}), since there is
$k\in m\setminus\nodes(N)$, \/ $\theta$ can be
expressed as a product $\sigma_0\sigma_1\ldots\sigma_t$ of maps such
that, for $s\leq t$, we have either $\sigma_s=Id_{-i}$ for some $i<m$
or $\sigma_s=[i/j]$ for some $i, j<m$ and where
$i\not\in\nodes(N\sigma_0\ldots\sigma_{s-1})$.
But clearly  $(N Id_{-j})^+\geq N^+$ and if $i\not\in\nodes(N)$ and $j\in\nodes(N)$, then
$N^+\neq 0 \rightarrow {(N[i/j])}^+\neq 0$.
The required now follows.  The last part is straightforward.

Using the above proven facts,  we are now ready to show that \pe\  has a \ws\ in $\bold G^m$. She can always
play a network $N$ with $\nodes(N)\subseteq m,$ such that
$N^+\neq 0$.\\
In the initial round, let \pa\ play $a\in \At\A$.
\pe\ plays a network $N$ with $N(0, \ldots, n-1)=a$. Then $N^+=a\neq 0$.
Recall that here \pa\ is offered only one (cylindrifier) move.
At a later stage, suppose \pa\ plays the cylindrifier move, which we denote by
$(N, \langle f_0, \ldots, f_{n-2}\rangle, k, b, l).$
He picks a previously played network $N$,  $f_i\in \nodes(N), \;l<n,  k\notin \{f_i: i<n-2\}$,
such that $b\leq {\sf c}_l N(f_0,\ldots,  f_{i-1}, x, f_{i+1}, \ldots, f_{n-2})$ and $N^+\neq 0$.
Let $\bar a=\langle f_0\ldots f_{i-1}, k, f_{i+1}, \ldots f_{n-2}\rangle.$
Then by  second part of  (3)  we have that ${\sf c}_lN^+\cdot {\sf s}_{\bar a}b\neq 0$
and so  by first part of (2), there is a network  $M$ such that
$M^+\cdot{\sf c}_{l}N^+\cdot {\sf s}_{\bar a}b\neq 0$.
Hence $M(f_0,\dots, f_{i-1}, k, f_{i-2}, \ldots$ $, f_{n-2})=b$,
$\nodes(M)=\nodes(N)\cup\set k$, and $M^+\neq 0$, so this property is maintained.
\end{proof}
 
We let ${\sf LCA}_n$ denote the elementary class of ${\sf RCA}_n$s satisfying the Lyndon conditions \cite[Definition 3.5.1]{HHbook2} It is known that ${\sf LCA}_n={\bf El}{\sf CRCA}_n$ where ${\bf El}$ denotes `elementary closure'.
\begin{theorem}\label{square} Let $2<n<m\leq \omega$.
${\bf El}\Nr_n\CA_{\omega}\cap {\bf At}\subsetneq {\sf LCA}_n$. Furthermore, 
for any elementary class $\sf K$ between ${\bf El}\Nr_n\CA_{\omega}\cap \bf At$ and ${\sf LCA}_n$, ${\sf RCA}_n$ is generated by 
$\At\sf K$,
\end{theorem}

\begin{proof} Throughout the proof fix $2<n<\omega$. 

It suffices to show that ${\sf Nr}_n\CA_{\omega}\cap {\bf At}\subseteq {\sf LCA}_n$, since the last class is elementary. This follows from Lemma \ref{n}, 
since if $\A\in \Nr_n\CA_{\omega}$ is atomic, then \pe\ has a \ws\ in $\bold G^{\omega}(\At\A)$, hence in $G_{\omega}(\At\A)$, {\it a fortiori}, \pe\ has a \ws\ 
in $G_k(\At\A)$ for all $k<\omega$, so (by definition) $\A\in {\sf LCA}_n$. 
To show strictness of the last inclusion, let $V={}^n\Q$ and let ${\A}\in {\sf Cs}_n$ have universe $\wp(V)$.
Then $\A\in {\sf Nr}_{n}\CA_{\omega}$.  
Let 
$y=\{s\in V: s_0+1=\sum_{i>0} s_i\}$ and ${\E}=\Sg^{\A}(\{y\}\cup X)$, where $X=\{\{s\}: s\in V\}$. 
Now $\E$ and $\A$ having same top element $V$, share the same atom structure, namely, the singletons, so 
$\Cm\At\E=\A$. Furthermore, plainly $\A, \E\in {\sf CRCA}_n$. 
So $\E\in {\sf CRCA}_n\subseteq {\sf LCA}_n$, 
and as proved in \cite{SL}, $\E\notin {\bf  El}\Nr_{n}{\sf CA}_{n+1}$, hence $\E$ witnesses the required strict inclusion.

Now we show that $\At{\bf El}\Nr_n\CA_{\omega}$ generates $\RCA_n$.
Let ${\sf FCs}_n$ denote the class of {\it full} $\Cs_n$s, that is ${\sf Cs}_n$s  
having universe $\wp(^nU)$
($U$ non--empty set). 
First we show that ${\sf FCs}_n\subseteq \Cm\At\Nr_n\CA_{\omega}$.
Let $\A\in {\sf FCs}_n$.  Then $\A\in \Nr_n\CA_{\omega}\cap \bf At$, hence $\At\A\in \At\Nr_n\CA_{\omega}$ 
and $\A=\Cm\At\A\in \Cm\At\Nr_n\CA_{\omega}$.
The required now follows from the following chain of inclusions: 
$\RCA_n={\bf SP}{\sf FCs}_n\subseteq {\bf SP}\Cm\At(\Nr_n\CA_{\omega})\subseteq {\bf SP}\Cm\At({\bf El}\Nr_n\CA_{\omega})\subseteq 
{\bf SP}\Cm \At{\sf K}\subseteq {\bf SP}\Cm{\sf LCAS}_n\subseteq  {\sf RCA}_n,$ where $\sf K$ is given above.
\end{proof}

\begin{theorem}\label{bsl} Let $\kappa$ be an infinite cardinal. Then there exists an atomless $\C\in \CA_{\omega}$ such that  for all 
$2<n<\omega$, $\Nrr_n\C$ is atomic, with $|\At(\mathfrak{Nr}_n\C)|=2^{\kappa}$, $\mathfrak{Nr}_n\C\in {\sf LCA}_n$, 
but $\mathfrak{Nr}_n\C$ is not completely representable. 
\end{theorem}

\begin{proof}
We use the following uncountable version of Ramsey's theorem due to
Erdos and Rado:
If $r\geq 2$ is finite, $k$  an infinite cardinal, then
$exp_r(k)^+\to (k^+)_k^{r+1}$
where $exp_0(k)=k$ and inductively $exp_{r+1}(k)=2^{exp_r(k)}$.
The above partition symbol describes the following statement. If $f$ is a coloring of the $r+1$
element subsets of a set of cardinality $exp_r(k)^+$
in $k$ many colors, then there is a homogeneous set of cardinality $k^+$
(a set, all whose $r+1$ element subsets get the same $f$-value).
Let $\kappa$ be the given cardinal. We use a simplified more basic version of a rainbow construction where only 
the two predominent  colours, namely, the reds and blues are available. 
The algebra $\C$ will be constructed from a relation algebra possesing an $\omega$-dimensional cylindric basis.
To define the relation algebra we specify its atoms and the forbidden triples of atoms. The atoms are $\Id, \; \g_0^i:i<2^{\kappa}$ and $\r_j:1\leq j<
\kappa$, all symmetric.  The forbidden triples of atoms are all
permutations of $({\sf Id}, x, y)$ for $x \neq y$, \/$(\r_j, \r_j, \r_j)$ for
$1\leq j<\kappa$ and $(\g_0^i, \g_0^{i'}, \g_0^{i^*})$ for $i, i',
i^*<2^{\kappa}.$ 
Write $\g_0$ for $\set{\g_0^i:i<2^{\kappa}}$ and $\r_+$ for
$\set{\r_j:1\leq j<\kappa}$. Call this atom
structure $\alpha$.  
Consider the term algebra $\A$ defined to be the subalgebra of the complex algebra of this atom structure generated by the atoms.
We claim that $\A$, as a relation algebra,  has no complete representation, hence any algebra sharing this 
atom structure is not completely representable, too. Indeed, it is easy to show that if $\A$ and $\B$ 
are atomic relation algebras sharing the same atom structure, so that $\At\A=\At\B$, then $\A$ is completely representable $\iff$ $\B$ is completely representable.
Assume for contradiction that $\A$ has a complete representation $\Mo$.  Let $x, y$ be points in the
representation with $\Mo \models \r_1(x, y)$.  For each $i< 2^{\kappa}$, there is a
point $z_i \in \Mo$ such that $\Mo \models \g_0^i(x, z_i) \wedge \r_1(z_i, y)$.
Let $Z = \set{z_i:i<2^{\kappa}}$.  Within $Z$, each edge is labelled by one of the $\kappa$ atoms in
$\r_+$.  The Erdos-Rado theorem forces the existence of three points
$z^1, z^2, z^3 \in Z$ such that $\Mo \models \r_j(z^1, z^2) \wedge \r_j(z^2, z^3)
\wedge \r_j(z^3, z_1)$, for some single $j<\kappa$.  This contradicts the
definition of composition in $\A$ (since we avoided monochromatic triangles).
Let $S$ be the set of all atomic $\A$-networks $N$ with nodes
$\omega$ such that $\{\r_i: 1\leq i<\kappa: \r_i \text{ is the label
of an edge in $N$}\}$ is finite.
Then it is straightforward to show $S$ is an amalgamation class, that is for all $M, N
\in S$ if $M \equiv_{ij} N$ then there is $L \in S$ with
$M \equiv_i L \equiv_j N$, witness \cite[Definition 12.8]{HHbook} for notation.
Now let $X$ be the set of finite $\A$-networks $N$ with nodes
$\subseteq\kappa$ such that:

\begin{enumerate}
\item each edge of $N$ is either (a) an atom of
$\A$ or (b) a cofinite subset of $\r_+=\set{\r_j:1\leq j<\kappa}$ or (c)
a cofinite subset of $\g_0=\set{\g_0^i:i<2^{\kappa}}$ and

\item  $N$ is `triangle-closed', i.e. for all $l, m, n \in \nodes(N)$ we
have $N(l, n) \leq N(l,m);N(m,n)$.  That means if an edge $(l,m)$ is
labelled by $\sf Id$ then $N(l,n)= N(m,n)$ and if $N(l,m), N(m,n) \leq
\g_0$ then $N(l,n)\cdot \g_0 = 0$ and if $N(l,m)=N(m,n) =
\r_j$ (some $1\leq j<\omega$) then $N(l,n)\cdot \r_j = 0$.
\end{enumerate}
For $N\in X$ let $\widehat{N}\in\Ca(S)$ be defined by
$$\set{L\in S: L(m,n)\leq
N(m,n) \mbox{ for } m,n\in \nodes(N)}.$$
For $i\in \omega$, let $N\restr{-i}$ be the subgraph of $N$ obtained by deleting the node $i$.
Then if $N\in X, \; i<\omega$ then $\widehat{\cyl i N} =
\widehat{N\restr{-i}}$.
The inclusion $\widehat{\cyl i N} \subseteq (\widehat{N\restr{-i})}$ is clear.
Conversely, let $L \in \widehat{(N\restr{-i})}$.  We seek $M \equiv_i L$ with
$M\in \widehat{N}$.  This will prove that $L \in \widehat{\cyl i N}$, as required.
Since $L\in S$ the set $T = \set{\r_i \notin L}$ is infinite.  Let $T$
be the disjoint union of two infinite sets $Y \cup Y'$, say.  To
define the $\omega$-network $M$ we must define the labels of all edges
involving the node $i$ (other labels are given by $M\equiv_i L$).  We
define these labels by enumerating the edges and labeling them one at
a time.  So let $j \neq i < \kappa$.  Suppose $j\in \nodes(N)$.  We
must choose $M(i,j) \leq N(i,j)$.  If $N(i,j)$ is an atom then of
course $M(i,j)=N(i,j)$.  Since $N$ is finite, this defines only
finitely many labels of $M$.  If $N(i,j)$ is a cofinite subset of
$\g_0$ then we let $M(i,j)$ be an arbitrary atom in $N(i,j)$.  And if
$N(i,j)$ is a cofinite subset of $\r_+$ then let $M(i,j)$ be an element
of $N(i,j)\cap Y$ which has not been used as the label of any edge of
$M$ which has already been chosen (possible, since at each stage only
finitely many have been chosen so far).  If $j\notin \nodes(N)$ then we
can let $M(i,j)= \r_k \in Y$ some $1\leq k < \kappa$ such that no edge of $M$
has already been labelled by $\r_k$.  It is not hard to check that each
triangle of $M$ is consistent (we have avoided all monochromatic
triangles) and clearly $M\in \widehat{N}$ and $M\equiv_i L$.  The labeling avoided all
but finitely many elements of $Y'$, so $M\in S$. So
$\widehat{(N\restr{-i})} \subseteq \widehat{\cyl i N}$.

Now let $\widehat{X} = \set{\widehat{N}:N\in X} \subseteq \Ca(S)$.
Then we claim that the subalgebra of $\Ca(S)$ generated by $\widehat{X}$ is simply obtained from
$\widehat{X}$ by closing under finite unions.
Clearly all these finite unions are generated by $\widehat{X}$.  We must show
that the set of finite unions of $\widehat{X}$ is closed under all cylindric
operations.  Closure under unions is given.  For $\widehat{N}\in X$ we have
$-\widehat{N} = \bigcup_{m,n\in \nodes(N)}\widehat{N_{mn}}$ where $N_{mn}$ is a network
with nodes $\set{m,n}$ and labeling $N_{mn}(m,n) = -N(m,n)$. $N_{mn}$
may not belong to $X$ but it is equivalent to a union of at most finitely many
members of $\widehat{X}$.  The diagonal $\diag ij \in\Ca(S)$ is equal to $\widehat{N}$
where $N$ is a network with nodes $\set{i,j}$ and labeling
$N(i,j)=\sf Id$.  Closure under cylindrification is given.
Let $\C$ be the subalgebra of $\Ca(S)$ generated by $\widehat{X}$.
Then $\A = \mathfrak{Ra}(\C)$.
To see why, each element of $\A$ is a union of a finite number of atoms,
possibly a co--finite subset of $\g_0$ and possibly a co--finite subset
of $\r_+$.  Clearly $\A\subseteq\mathfrak{Ra}(\C)$.  Conversely, each element
$z \in \mathfrak{Ra}(\C)$ is a finite union $\bigcup_{N\in F}\widehat{N}$, for some
finite subset $F$ of $X$, satisfying $\cyl i z = z$, for $i > 1$. Let $i_0,
\ldots, i_k$ be an enumeration of all the nodes, other than $0$ and
$1$, that occur as nodes of networks in $F$.  Then, $\cyl
{i_0} \ldots
\cyl {i_k}z = \bigcup_{N\in F} \cyl {i_0} \ldots
\cyl {i_k}\widehat{N} = \bigcup_{N\in F} \widehat{(N\restr{\set{0,1}})} \in \A$.  So $\mathfrak{Ra}(\C)
\subseteq \A$.
$\A$ is relation algebra reduct of $\C\in\CA_\omega$ but has no complete representation.
Let $n>2$. Let $\B=\Nrr_n \C$. Then
$\B\in {\sf Nr}_n\CA_{\omega}$, is atomic, but has no complete representation for plainly a complete representation of $\B$ induces one of $\A$. 
In fact, because $\B$  is generated by its two dimensional elements,
and its dimension is at least three, its
$\Df$ reduct is not completely representable.

It remains to show that the $\omega$--dilation $\C$ is atomless. 
For any $N\in X$, we can add an extra node 
extending
$N$ to $M$ such that $\emptyset\subsetneq M'\subsetneq N'$, so that $N'$ cannot be an atom in $\C$.
\end{proof}
%

\subsection{Complete and other notions of representability}
In the previous construction used in Proposition \ref{bsl}, $\A$ also satisfies the Lyndon conditions by \cite[Theorem 33]{r}
but is not completely representable. Thus:
\begin{corollary}\label{HH} The class $\sf CRRA$ is  not elementary.
\end{corollary}
But we can go further:
\begin{theorem}\label{raa} 
The class
$\sf CRRA$ is not closed under $\equiv_{\infty, \omega}$. 
\end{theorem}
\begin{proof} 

Take $\R$ to be a symmetric, atomic relation algebra with atoms
$$\Id, \r(i),
\y(i), \bb(i):i<\omega.$$
Non-identity atoms have colours, $\r$ is red,
$\bb$ is blue, and $\y$ is yellow. All atoms are self-converse.
Composition of atoms is defined
by listing the forbidden triples.
The forbidden triples are (Peircean transforms)
or permutations of $(\Id, x, y)$ for $x\neq y$, and
$$(\r(i), \r(i), \r(j)), \; (\y(i), \y(i), \y(j)), \; (\bb(i), \bb(i), \bb(j))\; \; i\leq j < \omega$$
$\R$ is the complex algebra over this atom structure.
Let $\alpha$ be an ordinal.  $\R^\alpha$ is obtained from $\R$ by
splitting the atom $\r(0)$ into $\alpha$ parts $\r^k(0):k<\alpha$
and then taking the full complex algebra.
In more detail, we put red atoms $r^{k}(0)$ for $k<\alpha.$
In the altered algebra the forbidden triples are
$(\y(i), \y(i), \y(j)), (\bb(i), \bb(i), \bb(j)), \ \   i\leq j<\omega,$
$(\r(i), \r(i), \r(j)),  \ \  0<i\leq j<\omega,$
$(\r^k(0), \r^l(0), \r(j)), \ \   0<j<\omega, k,l<\alpha,$
$(\r^k(0), \r^l(0), \r^m(0)), \ \  k,l,m<\alpha.$
Now let  $\B = \R^\omega$ and $\A=\R^{\mathfrak{n}}$ with $\mathfrak{n}\geq 2^{\aleph_0}$.

For an ordinal $\alpha$,  $\R^{\alpha}$ is as defined in the previous remark.
In $\R^\alpha$, we use the
following abbreviations:
$r(0) = \sum_{k<\alpha}\r^k(0)$
$\r = \sum_{i<\omega}\r(i)$
$\y = \sum_{i<\omega}\y(i)$
$\bb = \sum_{i<\omega}\bb(i).$
These suprema exist because they are taken in the complex algebras which are complete.
The \emph{index} of $\r(i), \y(i)$ and $\bb(i)$ is $i$ and the index of
$\r^k(0)$ is also $0$.
Now let  $\B = \R^\omega$ and $\A=\R^{\mathfrak{n}}$ with $\mathfrak{n}\geq 2^{\aleph_0}$. We claim that
$\B\in \Ra\CA_{\omega}$ and $\A\equiv \B$.
For the first required, we show that $\B$ has a cylindric bases by exhibiting a \ws\ for  \pe\
in the the cylindric-basis game, which is a simpler version of the hyperbasis game
\cite[Definition 12.26]{HHbook}. 
At some stage of the game, let the play so far be
$N_0, N_1, \ldots, N_{t-1}$ for some $t<\omega$.
We say that an edge $(m,n)$ of an atomic network $N$ is a
\emph{diversity edge} if $N(m,n)\cdot   \Id =0$.  Each diversity edge of
each atomic network in the play has an owner --- either \pe\ or \pa,
which we will allocate as we define \pe's strategy.  If an edge
$(m,n)$ belongs to player $p$ then so does the reverse edge $(n,m)$
and we will only specify one of them.  Since our algebra is symmetric, so
the label of the reverse edge is equal to the label of the edge, so again need to specify only one.
For the next round \pe\ must define $N_t$ in response to \pa's move.
If there is an already played network $N_i$
(some $i<t$) and a finitary map $\sigma:\omega\to\omega$ such that
$N_t\sigma$ `answers' his move, then she lets $N_t= N_i\sigma$.
From now on we assume that there is no such $N_i$ and $\sigma$.  We
consider the three types of \pa\ can make.  If he plays an {\it atom move}
by picking an atom $a$, \/ \pe\ plays an atomic network $N$ with
$N(0,1) = a$ and for all $x \in\omega\setminus\set{1}$, $N(0,x) = \Id$.

If \pa\ plays a {\it triangle move} by picking a previously played $N_x$
(some $x<t$), nodes $i, j, k$ with $k\notin\set{i,j}$ and atoms $a, b$
with $a;b \geq N_x(i,j)$, we know that $a, b \neq 1'$, as we are
assuming the \pe\ cannot play an embedding move (if $a = \Id$, consider
$N_x$ and the map $[k/i]$).  \pe\ must
play a network $N_t\equiv_k N_x$ such that $N_t(i, k)= a, \; N_t(k,
j) = b$.  These edges, $(i,k)$ and $(k,j)$, belong to \pa\ in $N_t$.  All
diversity edges not involving $k$ have the same owner in $N_t$ as they
did in $N_x$.  And all edges $(l, k)$ for $k\notin\set{i, j}$ belong
to \pe\ in $N_x$. To label these edges
\pe\ chooses a colour $c$ different than the
colours of $a, b$(we have three colours so this is possible).  Then, one at a time, she labels each edge $(l, k)$
by an atom with colour $c$ and a non-zero index which has not yet been
used to label any edge of any network played in the game.  She does
this one edge at a time, each with a new index.  There are infinitely
many indices to choose, so this can be done.

Finally, \pa\ can play an amalgamation move by picking $M, N \in
\set{N_s:s<t}$, nodes $i, j$ such that $M\equiv_{ij} N$.  If there is
$N_s$ (some $s<t$) and a map $\sigma: \nodes(N_s) \to \nodes(M)\cup
\nodes(N)$ such that $M\equiv_i N_s\sigma\equiv_j N$ then \pe\ lets
$N_t= N_s\sigma$.  Ownership of edges is inherited from $N_s$.  If
there is no such $N_s$ and $\sigma$ then there are two cases.  If
there are three nodes $x, y, z$ in the `amalgam' such that $M(j, x)$
and $N(x, i)$ are both red and of the same index, $M(j, y), N(y, i)$
are both yellow and of the same index and $M(j, z), N(z, i)$ are both
blue and of the same index, then the new edge $(i, j)$ belongs to \pa\
in $N_t$.  It will be labelled by either
$\r^0(0), \bb(0)$ or $\y(0)$ and it it is easy to show that at least one
of these will be a consistent choice.  Otherwise, if there is no such
$x, y, z$ then the new edge $(i, j)$ belongs to \pe\ in $N_t$.  She
chooses a colour $c$ such that there is no $x$ with $M(j, x)$ and
$N(x, i)$ both having colour $c$ and the same index.  And she chooses
a non-zero index for $N_t(i, j)$ which is new to the game (as with
triangle moves).  If $k\neq k' \in M\cap N$ then $(j, k)$ has the same owner
in $N_t$ as it does in $M$, $(k, i)$ has the same owner in $N_t$ as it does in
$N$ and $(k, k')$ belongs to \pe\ in $N_t$ if it belongs to \pe\ in
\emph{either} $M$ or $N$, otherwise it belongs to \pa\ in $N_t$.
Now the only way \pe\ could lose, is if \pa\ played an amalgamation
move $(M, N, i, j)$ such that there are $x, y, z \in M \cap N$ such
that $M(j, x) = \r^k(0), \; N(x,i) = \r^{k'}(0), \; M(j, y) = N(y, i)
= \bb(0)$ and $M(j,z) = N(z, i) = \y(0)$.  But according to \pe's
strategy, she never chooses atoms with index $0$, so all these edges
must have been chosen by \pa.
This contradiction proves the required.

Now, let $\H$ be an $\omega$-dimensional cylindric basis for $\B$. Then  
$\Ca\H\in \CA_{\omega}$.  Consider the
cylindric algebra $\C = \Sg^{\Ca\H}\B$, the subalgebra of $\Ca\H$ generated by $\B$.  In principal, new two dimensional elements that 
were not originally in $\B$,  
can be created in $\C$ using the spare dimensions in $\Ca(\H)$.
But next we exclude this possibility. We show that $\B$ exhausts the $2$--dimensional elements of  $\Ra\C$, more concisely, 
we show that 
$\B=\Ra\C$. For this purpose we want to find out what are the elements of
$\Ca\H$ that are generated by $\B$.
Let $M$ be a (not necessarily atomic) finite network over $\B$ whose nodes are a
finite subset of $\omega$.
\begin{itemize}
\item   Define (using the same noation in the proof of Theorem \ref{bsl})
$\widehat{M} = \set{N\in\H: N\leq M}\in\Ca\H$.  ($N\leq M$ means that
for all $i, j\in M$ we have $N(i, j)\leq M(i, j)$.)

\item A \emph{block} is an element of the form $\widehat{M}$ for some
finite network $M$ such that
\begin{enumerate}
\item $M$ is triangle-closed, i.e. for all $i, j, k\in M$ we have
$M(i, k)\leq M(i,j);M(j,k)$
\item If $x$ is the label of an irreflexive edge of $M$ then $x={\sf Id}$ or
$x\leq \r$ or $x\leq \y$ or $x\leq \bb$ (we say $x$ is
`monochromatic'), and $|\set{i:x\cdot (\r(i)+\y(i)+\bb(i))\neq 0}|$ is
either $0, 1$ or infinite (we say that the number of indices of $x$ is
either $0, 1$ or infinite).
\end{enumerate}
\end{itemize}
We prove:
\begin{enumerate}
\item For any block $\widehat{M}$ and $i<\omega$ we have
\[\cyl i\widehat{M} = (M\restr{\dom(M)\setminus\set{i}})\widehat{\;\;}\]
\item
The domain of $\C$ consists of finite sums of blocks.
\end{enumerate}
$\cyl i
\widehat{M}\subseteq(M\restr{\dom(M)\setminus\set{i}})\widehat{\;\;}$ is
obvious.  If $i\notin M$ the equality is trivial.  Let $N \in
(M\restr{\dom(M)\setminus\set{i}})\widehat{\;\;}$, i.e. $N\leq
M\restr{\dom(M)\setminus\set{i}}$.  We must show that $N\in\cyl
i\widehat{M}$ and for this we must find $L \equiv_i N$ with
$L\in\widehat{M}$.  $L \equiv_i N$ determines every edge of $L$ except
those involving $i$.  For each $j\in M$, if the number of indices in
$M(i,j)$ is just one, say $M(i,j) = \r(k)$, then let $L(i,j)$ be an
arbitrary atom below $\r(k)$.  There should be no inconsistencies in
the labelling so far defined for $L$, by triangle-closure for $M$.
For all the other edges $(i, j)$ if $j \in M$ there are infinitely
many indices in $M(i,j)$ and if $j\notin M$ then we have an
unrestricted choice of atoms for the label.  These edges are labelled
one at a time and each label is given an atom with a new index, thus
avoiding any inconsistencies.  This defines $L \equiv_i N$ with $L\in
\widehat{M}$. For the second part, we already have seen that the set of finite sums
of blocks is closed under cylindrification.  We'll show that this set
is closed under all the cylindric operations and includes $\B$.  For
any $x\in \B$ and $i,j<\omega$, let $N^{ij}_x$ be the $\B$-network
with two nodes $\set{i,j}$ and labelling $N^{ij}_x(i,i) = N^{ij}_x(j,j) = \Id$,
and $N^{ij}(i,j) = x, \; N^{ij}_x(j,i) = \breve{x}$.  Clearly $N^{ij}_x$ is
triangle closed.  And $\widehat{N^{01}_x} = x$.  For any $x \in \B$,
we have $x = x\cdot {\sf Id}+x\cdot \r + x\cdot \y + x\cdot \bb$, so $x
=\widehat{N^{01}_{x\cdot \Id}}+ \widehat{N^{01}_{x\cdot \r}}
+\widehat{N^{01}_{x\cdot \y}} + \widehat{N^{01}_{x\cdot \bb}}$ and the labels of
these four networks are monochromatic.  The first network defines a
block and for each of the last three, if the number if indices is
infinite then it is a block.  If the number of indices is finite then
it is a finite union of blocks.  So every element of $\B$ is a
finite union of blocks.

For the diagonal elements, $\diag{i}{j} = \widehat{N^{ij}_{\sf Id}}$.
Closure under sums is obvious.
For negation, take a block $\widehat{M}$. Then
$ -\widehat{M} = \sum_{i,j\in M} \widehat{N^{ij}_{-N(i,j)}}.$
As before we can replace $\widehat{N^{ij}_{-N(i,j)}}$ by a finite
union of blocks. Thus the set of finite sums of blocks includes $\B$ and the
diagonals and is closed under all the cylindric operations.  Since
every block is clearly generated from $\B$ using substitutions and
intersection only.
It remains to show that $\B = \Ra\C$.  Take a block $\widehat{M}
\in \Ra\C$.  Then $\cyl i\widehat{M} = \widehat{M}$ for $2\leq i <
\omega$.  By the first part of the lemma, $\widehat{M} =
\widehat{M\restr{\set{0,1}}} \in \B$.\\
We finally show that \pe\ has a \ws\ in an \ef-game over $(\A, \B)$ concluding that $\A\equiv_{\infty}\B$.
At any stage of the game,
if \pa\ places a pebble on one of
$\A$ or $\B$, \pe\ must place a matching pebble,  on the other
algebra.  Let $\b a = \la{a_0, a_1, \ldots, a_{n-1}}$ be the position
of the pebbles played so far (by either player) on $\A$ and let $\b
b = \la{b_0, \ldots, b_{n-1}}$ be the the position of the pebbles played
on $\B$.  \pe\ maintains the following properties throughout the
game.
\begin{itemize}
\item For any atom $x$ (of either algebra) with
$x\cdot \r(0)=0$ then $x \in a_i\iff x\in b_i$.
\item $\b a$ induces a finite partion of $\r(0)$ in $\A$ of $2^n$
 (possibly empty) parts $p_i:i<2^n$ and $\b b$ induces a partion of
 $\r(0)$ in $\B$ of parts $q_i:i<2^n$.  $p_i$ is finite iff $q_i$ is
 finite and, in this case, $|p_i|=|q_i|$.
\end{itemize}

Now we show that $\sf CRRA$ is not closed under $\equiv_{\infty, \omega}$. 
Since $\B\in \Ra\CA_{\omega}$ has countably many atoms, 
then $\B$ is completely representable \cite[Theorem 29]{r}.
For this purpose, we show that $\A$ is not completely representable. We work with the term algebra, $\Tm\At\A$, since the latter is completely representable 
$\iff$ the complex algebra is.
Let  $\r = \{\r(i): 1\leq i<\omega\}\cup \{\r^k(0): k<2^{\aleph_0}\}$, $\y = \{\y(i):  i\in \omega\}$, $\bb^+ = \{\bb(i): i\in \omega\}.$
It is not hard to check every element of $\Tm\At\A\subseteq \wp(\At\A)$ has the form  
$F\cup R_0\cup B_0\cup Y_0$, where $F$ is a finite set of atoms, $R_0$ is either empty or a co-finite subset of $\r$, $B_0$ 
is either empty or a co--finite subset of $\bb$, and $Y_0$ is either empty or a co--finite subset 
of $\y$. 
Using an argument similar to that used in the proof of Theorem \ref{bsl}, we show  that the existence of a complete representation necessarily forces a 
monochromatic triangle, that we avoided at the start when defining $\A$.
Let $x, y$ be points in the
representation with $M \models \y(0)(x, y)$.  For each $i< 2^{\aleph_0}$, there is a
point $z_i \in M$ such that $M \models {\sf red}(x, z_i) \wedge \y(0)(z_i, y)$ (some red $\sf red\in \r$).
Let $Z = \set{z_i:i<2^{\aleph_0}}$.  Within $Z$ each edge is labelled by one of the $\omega$ atoms in
$\y^+$ or $\bb^+$.  The Erdos-Rado theorem forces the existence of three points
$z^1, z^2, z^3 \in Z$ such that $M \models \y(j)(z^1, z^2) \wedge \y(j)(z^2, z^3)
\wedge \y(j)(z^3, z_1)$, for some single $j<\omega$  
or three  points $z^1, z^2, z^3 \in Z$ such that $M \models \bb(l)(z^1, z^2) \wedge \bb(l)(z^2, z^3)
\wedge \bb(l)(z^3, z_1)$, for some single $l<\omega$.  
This contradicts the
definition of composition in $\A$ (since we avoided monochromatic triangles).
We have proved that $\sf CRRA$ is not closed under $\equiv_{\infty, \omega}$, since $\A\equiv_{\infty, \omega}\B$, 
$\A$ is not completely representable,   but $\B$ is completely representable.
\end{proof}

\begin{theorem} Any class $\sf K$, such that $\Ra\CA_{\omega}\subseteq \sf K\subseteq \Ra\CA_5$ is not elementary.
\end{theorem}
\begin{proof} Using the notation in the proof of the last theorem, where we proved that $\B\in \Ra\CA_{\omega}$ and $\A\equiv \B$; it therefore suffices to show that  
that $\A\notin \Ra\CA_5$. Let $\kappa=2^{\aleph_0}$. We use $\sum$ to denote suprema which 
exists in $\A$. 
Notation, cf. \cite{HH} 13.30. For $\b a\in \;^{<n}n$ we write $\sub{}{\b a}$ for an
arbitrary string of substitutions $w$ such that $\hat{w} = \b a$.
In more detail. Let $n\geq 3$ and $i,j<n$. We define a string of substitutions ${\sf s}_{ij}$:
$${\sf s}_{ij}={\sf s}_i^0{\sf s}_j^1, \text { if } j\neq 0$$
$${\sf s}_{ij}={\sf s}_0^1{\sf s}_i^0 \text { iff } j=0, i\neq 1$$
$${\sf s}_{ij}={\sf s}_0^2{\sf s}_1^0{\sf s}_2^1 \text { iff }j=0, i=1$$

We use that if $\C\in\CA_n, \; i,j,k<n,, \; k \neq i, j$.  Then
$\sub{}{ij}(r;s) = \cyl{k}(\sub{}{ik}r.\sub{}{kj}s)$, for all $r,
s\in\Ra\C$.
Suppose for contradiction that $\A = \Ra\C$ for some $\C\in\CA_5$.  
Then $\cyl2\r^0(0) =\r^0(0) \leq \r^k(0);\y(0) = \cyl2(\sub12\r^k(0)\cdot \sub02\y(0))$, for each $k<\kappa$.
Therefore
\[x_k= \r^0(0)\cdot \sub12\r^k(0)\cdot \sub02\y(0) \neq 0.\]
Then by a tedious computation  
one shows that 
$x_0\leq\sum_{i<\omega}\cyl3(\sub23x_k\cdot \sub{}{23}\bb(i))$.
This holds for all $0<k<\omega_1$.  So,
\begin{eqnarray*}
x_0&\leq&\prod_{0<k<\kappa}\sum_{i<\omega}\cyl3(\sub23x_k\cdot\sub02\sub13b(i))\\
 &=& \sum_{g:\kappa_1\setminus\set{0}\to\omega}\prod_{0<k<\omega_1}\cyl3(\sub23x_k\cdot\sub02\sub13b(g(k)))
\end{eqnarray*}
So there is a function $g:\kappa\setminus\set{0}\to\omega$ such that
\[x_0\; . \; \prod_{0<k<\omega_1}\cyl3(\sub23x_k\cdot \sub02\sub13\bb(g(k))) \neq 0\]
Pick $i<\omega$ such that $X=g^{-1}(i)$ is uncountable, we need an uncountable number of superscripts $k$ at this
point only.  Then
\[ \xi =x_0\; . \; \prod_{k\in X}\cyl3(\sub23x_k\cdot \sub02\sub13\bb(i)) \neq 0\]
Let
\[z_k = \sub23x_k\cdot \sub02\sub13\bb(i)\cdot \xi\]
 for each $k\in
X$.  Let $S_0 = \set{z_k:k\in X}$.  $S_0$ has the following properties.  There is an index $i<\omega$ such that for all $z, x \in S_0$,
\begin{enumerate}
\item $\cyl4z=z$ \label{prop:c4}
\item $\cyl3z = \cyl3x$ \label{prop:c3}
\item  \label{prop:rk}
\begin{enumerate}
\item $\exists k<\omega_1$ such that $z\leq\sub13\r^k(0)$
\item $\forall k<\omega_1$ if $z, x \leq\sub13\r^k(0)$ then $z=x$
\end{enumerate}
\item $z \leq\sub12\r^0(0)\cdot \sub02\y(0)\cdot \sub03\y(0)\cdot \sub02\sub13\bb(i)$  \label{prop:others}
\item $S_0$ is infinite.  \label{prop:inf}
\end{enumerate}
Suppose there is an infinite set $S$ and an index $i<\omega$
with the properties listed above.
We show how to construct another infinite set $S'$ and a new
index $i'<i$ with the same properties.  Iterating this
construction $i+1$ times will then lead to a contradiction
since the index cannot be less than $0$.
Fix $z\in S$.  For each $x\in S\setminus\set{z}$ and $j<i$,
let 
\[\tau^j_x = \cyl4(\sub34x\cdot \sub02\sub14\bb(j)).\]
Then by another tedious computation we have : (**)  for any $x\in S\setminus\set{z}$,
\[ z \leq \sum_{j<i}\tau^j_x \]

We now construct $S'$ and $i'$ from $S$ and $i$ with the
required properties.
By (**)
\[z\leq\prod_{x\in S\setminus{z}}\sum_{j<i}\tau^j_x
=\sum_{g:S\setminus\set{z}\to i}\prod_{x\in
S\setminus\set{z}}\tau^{g(x)}_x\]
Since $z\neq 0$ there is $g:S\setminus\set{z}\to i
\;\;(=\set{0, 1, \ldots, i-1})$ such that $z\cdot\prod_{x\in
S\setminus\set{z}}\tau^{g(x)}_x\neq 0$.  Pick $i'<i$ such
that $X=g^{-1}(i')$ is infinite.  Then
\[ z\cdot \prod_{x\in X}\tau^{i'}_x \neq 0\]
Define
\begin{eqnarray*}
\xi &=&z\cdot \prod_{x\in X}\tau^{i'}_x \;\;\;\neq 0\\
x'&=&\sub34x.\sub03\sub14\bb(i') . \xi \\
x" &=&\sub43\sub32\cyl2x'\\
S'&=&\set{x":x\in X}
\end{eqnarray*}
We check each of the properties.  Property~\ref{prop:c4} is obvious.
By property~\ref{prop:rk} for $S$, if $x\in X$ then there is
$k<\omega_1$ and $x\leq\sub13\r^k(0)$.  So
\begin{eqnarray*}x"&\leq&\sub43\sub32\cyl2\sub34\sub13\r^k(0
)\\
&=&\sub43\sub32\cyl2\sub14\sub31\r^k(0)\\
&=&\sub43\sub32\sub14\cyl2\r^k(0)\\
&=&\sub43\sub14\sub32\r^k(0)\\
&=& \sub13\sub41\r^k(0)\\
&=&\sub13\r^k(0)
\end{eqnarray*}
This gives property~\ref{prop:rk} for $S'$ and shows that $S'$ is
infinite (property~\ref{prop:inf}).
For property~\ref{prop:c3} we first prove that if $x\in S$ then
$\cyl4x' = \xi$.  First note that $\cyl4\xi = \xi$, so
\begin{eqnarray*}\cyl4x'& =& \cyl4(\sub34x\cdot \sub{}{34}\bb(i')\cdot \xi)\\
&=&\cyl4(\sub34x\cdot \sub{}{34}\bb(i'))\cdot \xi \\
&=&\tau^{i'}_x\cdot \xi\\
&=&\xi
\end{eqnarray*}
Hence, 
\begin{eqnarray*}\cyl3x'' &= &\cyl3\sub43\sub32\cyl2x'\\
 &=&\cyl4\sub43\sub32\cyl2x'\\
&=&\cyl4\sub32\cyl2x'\\
&=&\sub32\cyl2\cyl4x'\\
&=&\sub32\cyl2\xi
\end{eqnarray*}
which gives property~\ref{prop:c3}.
 Finally, for property~\ref{prop:others}, we must prove that
$x''\leq\sub12\r^0(0)\cdot\sub02\y(0)\cdot \sub03\y(0)\cdot \sub02\sub13\bb(i)$, for
each $x''\in S'$.  Property~\ref{prop:others} for $S$ says that
$x\leq\sub03\y(0)$ and since $x'\leq\sub34x$ we get
$x'\leq\sub04\y(0)$.  Therefore
\begin{eqnarray*}x''&\leq&\sub43\sub32\cyl2\sub{}{41}\y(0)\\
&=&\sub43\sub32\cyl2\sub04\y(0)\\
&=&\sub43\sub32\sub04\cyl2\y(0)\\
&=&\sub43\sub04\sub32\y(0)\\
&=&\sub03\sub40\y(0)\\
&=&\sub03\y(0)
\end{eqnarray*}
Similarly, we can show that $x''\leq\sub12\r^0(0)\cdot \sub02\y(0)$.  And
$x'\leq\sub03\sub14\bb(i')$ gives $x''\leq\sub43\sub32\cyl2\sub03\sub14\bb(i')\leq\sub02\sub13\bb(i')$.  This proves property~\ref{prop:others}
and we are done.
\end{proof}
\section{Complete representations and non-elementary classes}
In the construction used in Proposition \ref{bsl}, both $\R$ and $\Nr_n\C$ satisfy the Lyndon conditions 
but are not completely representable. Thus:
\begin{corollary}\cite{HH}\label{HH} Let $2<n<\omega$. Then the classes $\sf CRRA$ and ${\sf CRCA}_n$ are not elementary.
\end{corollary}

We next strengthen the last theorem. 
We first define  a game $\bold H$ that involves certain {\it hypernetworks}. A  $\lambda$--neat hypernetwork  is roughly 
a network endowed with hyperdeges of length $\neq n$ allowed to get arbitrarily long but are of finite length, and such hyperedges 
get their labels from a non--empty set of labels $\Lambda$; such that all so--called {\it short hyperedges} are constantly labelled by $\lambda\in \Lambda$.
The board of the game consists of  $\lambda$-neat hypernetworks: 

\begin{definition}\label{hypernetwork} For an $n$--dimensional atomic network $N$  on an atomic $\CA_n$ and for  $x,y\in \nodes(N)$, set  $x\sim y$ if
there exists $\bar{z}$ such that $N(x,y,\bar{z})\leq {\sf d}_{01}$.
Define the  equivalence relation $\sim$ over the set of all finite sequences over $\nodes(N)$ by
$\bar x\sim\bar y$ iff $|\bar x|=|\bar y|$ and $x_i\sim y_i$ for all
$i<|\bar x|$. (It can be easily checked that this indeed an equivalence relation).
A \emph{ hypernetwork} $N=(N^a, N^h)$ over an atomic $\CA_n$
consists of an $n$--dimensional  network $N^a$
together with a labelling function for hyperlabels $N^h:\;\;^{<
\omega}\!\nodes(N)\to\Lambda$ (some arbitrary set of hyperlabels $\Lambda$)
such that for $\bar x, \bar y\in\; ^{< \omega}\!\nodes(N)$
if $\bar x\sim\bar y \Rightarrow N^h(\bar x)=N^h(\bar y).$
If $|\bar x|=k\in \N$ and $N^h(\bar x)=\lambda$, then we say that $\lambda$ is
a $k$-ary hyperlabel. $\bar x$ is referred to as a $k$--ary hyperedge, or simply a hyperedge.
A hyperedge $\bar{x}\in {}^{<\omega}\nodes(N)$ is {\it short}, if there are $y_0,\ldots, y_{n-1}$
that are nodes in $N$, such that
$N(x_i, y_0, \bar{z})\leq {\sf d}_{01}$
or $\ldots N(x_i, y_{n-1},\bar{z})\leq {\sf d}_{01}$
for all $i<|x|$, for some (equivalently for all)
$\bar{z}.$ Otherwise, it is called {\it long.}
This game involves, besides the standard 
cylindrifier move,  
two new amalgamation moves.
Concerning his moves, this game with $m$ rounds ($m\leq \omega$), call it  $\bold H_m$, \pa\ can play a cylindrifier move, like before but now played on $\lambda$---
neat hypernetworks ($\lambda$ a constant label).
Also \pa\ can play a \emph{transformation move} by picking a
previously played hypernetwork $N$ and a partial, finite surjection
$\theta:\omega\to\nodes(N)$, this move is denoted $(N, \theta)$.  \pe's
response is mandatory. She must respond with $N\theta$.
Finally, \pa\ can play an
\emph{amalgamation move} by picking previously played hypernetworks
$M, N$ such that
$M\restr {\nodes(M)\cap\nodes(N)}=N\restr {\nodes(M)\cap\nodes(N)},$
and $\nodes(M)\cap\nodes(N)\neq \emptyset$.
This move is denoted $(M,
N).$
To make a legal response, \pe\ must play a $\lambda_0$--neat
hypernetwork $L$ extending $M$ and $N$, where
$\nodes(L)=\nodes(M)\cup\nodes(N)$.
\end{definition}

\begin{theorem}\label{gripneat} Let $\alpha$ be a countable atom structure. If \pe\ has a \ws\ in $\bold H_{\omega}(\alpha)$, 
then there exists a complete $\D\in \RCA_{\omega}$ such that 
$\Cm\alpha\cong \Nrr_n\D$ and $\alpha\cong \At\Nrr_n\D$. 
In particular, $\Cm\alpha\in \Nr_n\CA_{\omega}$ 
and $\alpha\in \At\Nr_{n}\CA_{\omega}$.  
\end{theorem} 
\begin{proof} 
Fix some $a\in\alpha$. The game $\bold H_{\omega}$ is designed so that using \pe\ s \ws\ in the game $\bold H_{\omega}(\alpha)$ 
one can define a
nested sequence $M_0\subseteq M_1,\ldots$ of $\lambda$--neat hypernetworks
where $M_0$ is \pe's response to the initial \pa-move $a$, such that:
If $M_r$ is in the sequence and $M_r(\bar{x})\leq {\sf c}_ia$ for an atom $a$ and some $i<n$,
then there is $s\geq r$ and $d\in\nodes(M_s)$
such that  $M_s(\bar{y})=a$,  $\bar{y}_i=d$ and $\bar{y}\equiv_i \bar{x}$.
In addition, if $M_r$ is in the sequence and $\theta$ is any partial
isomorphism of $M_r$, then there is $s\geq r$ and a
partial isomorphism $\theta^+$ of $M_s$ extending $\theta$ such that
$\rng(\theta^+)\supseteq\nodes(M_r)$ (This can be done using \pe's responses to amalgamation moves).
Now let $\M_a$ be the limit of this sequence, that is $\M_a=\bigcup M_i$, the labelling of $n-1$ tuples of nodes
by atoms, and hyperedges by hyperlabels done in the obvious way using the fact that the $M_i$s are nested.
Let $L$ be the signature with one $n$-ary relation for
each $b\in\alpha$, and one $k$--ary predicate symbol for
each $k$--ary hyperlabel $\lambda$.
{\it Now we work in $L_{\infty, \omega}.$}
For fixed $f_a\in\;^\omega\!\nodes(\M_a)$, let
$\U_a=\set{f\in\;^\omega\!\nodes(\M_a):\set{i<\omega:g(i)\neq
f_a(i)}\mbox{ is finite}}$.
We  make $\U_a$ into the base of an $L$ relativized structure 
${\cal M}_a$ like in \cite[Theorem 29]{r} except that we allow a clause for infinitary disjunctions.
In more detail,  for $b\in\alpha,\; l_0, \ldots, l_{n-1}, i_0 \ldots, i_{k-1}<\omega$, \/ $k$--ary hyperlabels $\lambda$,
and all $L$-formulas $\phi, \phi_i, \psi$, and $f\in U_a$:
\begin{eqnarray*}
{\cal M}_a, f\models b(x_{l_0}\ldots,  x_{l_{n-1}})&\iff&{\cal M}_a(f(l_0),\ldots,  f(l_{n-1}))=b,\\
{\cal M}_a, f\models\lambda(x_{i_0}, \ldots,x_{i_{k-1}})&\iff&  {\cal M}_a(f(i_0), \ldots,f(i_{k-1}))=\lambda,\\
{\cal M}_a, f\models\neg\phi&\iff&{\cal M}_a, f\not\models\phi,\\
{\cal M}_a, f\models (\bigvee_{i\in I} \phi_i)&\iff&(\exists i\in I)({\cal M}_a,  f\models\phi_i),\\
{\cal M}_a, f\models\exists x_i\phi&\iff& {\cal M}_a, f[i/m]\models\phi, \mbox{ some }m\in\nodes({\cal M}_a).
\end{eqnarray*}
For any such $L$-formula $\phi$, write $\phi^{{\cal M}_a}$ for
$\set{f\in\U_a: {\cal M}_a, f\models\phi}.$
Let
$D_a= \set{\phi^{{\cal M}_a}:\phi\mbox{ is an $L$-formula}}$ and
$\D_a$ be the weak set algebra with universe $D_a$. 
Let $\D=\bold P_{a\in \alpha} \D_a$. Then $\D$ is a  generalized {\it complete} weak set algebra \cite[Definition 3.1.2 (iv)]{HMT2}.
Now we show that $\alpha\cong \At\mathfrak{Nr}_n\D$ and $\Cm\alpha\cong \mathfrak{Nr}_n\D$.
Let $x\in \D$. Then $x=(x_a:a\in\alpha)$, where $x_a\in\D_a$.  For $b\in\alpha$ let
$\pi_b:\D\to \D_b$ be the projection map defined by
$\pi_b(x_a:a\in\alpha) = x_b$.  Conversely, let $\iota_a:\D_a\to \D$
be the embedding defined by $\iota_a(y)=(x_b:b\in\alpha)$, where
$x_a=y$ and $x_b=0$ for $b\neq a$.  
Suppose $x\in\Nrr_n\D\setminus\set0$.  Since $x\neq 0$,
then it has a non-zero component  $\pi_a(x)\in\D_a$, for some $a\in \alpha$.
Assume that $\emptyset\neq\phi(x_{i_0}, \ldots, x_{i_{k-1}})^{\D_a}= \pi_a(x)$, for some $L$-formula $\phi(x_{i_0},\ldots, x_{i_{k-1}})$.  We
have $\phi(x_{i_0},\ldots, x_{i_{k-1}})^{\D_a}\in\Nrr_{n}\D_a$.
Pick
$f\in \phi(x_{i_0},\ldots, x_{i_{k-1}})^{\D_a}$  
and assume that ${\cal M}_a, f\models b(x_0,\ldots x_{n-1})$ for some $b\in \alpha$.
We show that
$b(x_0, x_1, \ldots, x_{n-1})^{\D_a}\subseteq
 \phi(x_{i_0},\ldots, x_{i_{k-1}})^{\D_a}$.  
Take any $g\in
b(x_0, x_1\ldots, x_{n-1})^{\D_a}$,
so that ${\cal M}_a, g\models b(x_0, \ldots x_{n-1})$.  
The map $\{(f(i), g(i)): i<n\}$
is a partial isomorphism of ${\cal M}_a.$ Here that short hyperedges are constantly labelled by $\lambda$ 
is used.
This map extends to a finite partial isomorphism
$\theta$ of $M_a$ whose domain includes $f(i_0), \ldots, f(i_{k-1})$.
Let $g'\in {\cal M}_a$ be defined by
\[ g'(i) =\left\{\begin{array}{ll}\theta(i)&\mbox{if }i\in\dom(\theta)\\
g(i)&\mbox{otherwise}\end{array}\right.\] 
We have ${\cal M}_a,
g'\models\phi(x_{i_0}, \ldots, x_{i_{k-1}})$. But 
$g'(0)=\theta(0)=g(0)$ and similarly $g'(n-1)=g(n-1)$, so $g$ is identical
to $g'$ over $n$ and it differs from $g'$ on only a finite
set.  Since $\phi(x_{i_0}, \ldots, x_{i_{k-1}})^{\D_a}\in\Nrr_{n}\D_a$, we get that
${\cal M}_a, g \models \phi(x_{i_0}, \ldots,
x_{i_k})$, so $g\in\phi(x_{i_0}, \ldots, x_{i_{k-1}})^{\D_a}$ (this can be proved by induction on quantifier depth of formulas).  
This
proves that 
$$b(x_0, x_1\ldots x_{n-1})^{\D_a}\subseteq\phi(x_{i_0},\ldots,
x_{i_k})^{\D_a}=\pi_a(x),$$ and so
$$\iota_a(b(x_0, x_1,\ldots x_{n-1})^{\D_a})\leq
\iota_a(\phi(x_{i_0},\ldots, x_{i_{k-1}})^{\D_a})\leq x\in\D_a\setminus\set0.$$
Now every non--zero element 
$x$ of $\Nrr_{n}\D_a$ is above a non--zero element of the following form 
$\iota_a(b(x_0, x_1,\ldots, x_{n-1})^{\D_a})$
(some $a, b\in \alpha$) and these are the atoms of $\Nrr_{n}\D_a$.  
The map defined  via $b \mapsto (b(x_0, x_1,\dots, x_{n-1})^{\D_a}:a\in \alpha)$ 
is an isomorphism of atom structures, 
so that $\alpha\in \At{\sf Nr}_n\CA_{\omega}$.  
Let  $X\subseteq \mathfrak{Nr}_n\D$. Then by completeness of $\D$, we get that
$d=\sum^{\D}X$ exists.  Assume that  $i\notin n$, then
${\sf c}_id={\sf c}_i\sum X=\sum_{x\in X}{\sf c}_ix=\sum X=d,$
because the ${\sf c}_i$s are completely additive and ${\sf c}_ix=x,$
for all $i\notin n$, since $x\in \mathfrak{Nr}_n\D$.
We conclude that $d\in \mathfrak{Nr}_n\D$, hence $d$ is an upper bound of $X$ in $\mathfrak{Nr}_n\D$. Since 
$d=\sum_{x\in X}^{\D}X$ there can be no $b\in \mathfrak{Nr}_n\D$ $(\subseteq \D)$ with $b<d$ such that $b$ is an upper bound of $X$ for else it will be an upper bound of $X$ in $\D$. 
Thus $\sum_{x\in X}^{\mathfrak{Nr}_n\D}X=d$ 
We have shown that   
$\mathfrak{Nr}_n\D$ is complete. 
Making the legitimate identification  
$\mathfrak{Nr}_n\D\subseteq_d \Cm\alpha$ by density,
we get that  $\mathfrak{Nr}_n\D=\Cm\alpha$ 
(since $\mathfrak{Nr}_n\D$ is complete),  
hence $\Cm\alpha\in {\sf Nr}_n\CA_{\omega}$. 
\end{proof}

If $\B$ is a Boolean algebra and $b\in \B$, 
then $\Rl_b\B$ denotes the Boolean algebra with domain $\{x\in B: x\leq b\}$, top element $b$, and other Boolean operations those of $\B$ relativized to $b$.
\begin{lemma}\label{join} 
In the following $\A$ and $\D$ are Boolean algebras.
\begin{enumerate}
  
\item  If $\A$  is atomic  and $0\neq a\in \A$, then $\Rl_a\A$ is also atomic. 
If $\A\subseteq_d \D$, and $a\in A$, then $\Rl_a\A\subseteq_d \Rl_a\D$,

\item  If $\A\subseteq_d \D$ then $\A\subseteq_c \D$. In particular, for any class $\sf K$ of $\sf BAOs$, ${\sf K}\subseteq \bold S_d{\sf K}\subseteq \bold  S_c{\sf K}$. 
If furthermore $\A$ and $\D$ are atomic,  then $\At\D\subseteq \At\A$. 
\end{enumerate}
\end{lemma}
\begin{proof}
(1): Let $b\in \Rl_a\D$ be non--zero. Then $b\leq a$ and $b$ is non-zero in $\D$. By atomicity of $\D$ there is an atom $c$ of $\D$ such that $c\leq b$. 
So $c\leq b\leq a$, thus $c\in \Rl_a\D$. Also $c$ is an atom in $\Rl_a\D$ because if not, then it will not be an atom in $\D$.  
The second part is similar.

(2): Assume that $\sum^{\A}S=1$ and for contradiction that there exists $b'\in \D$, $b'<1$ such that 
$s\leq b'$ for all $s\in S$. Let $b=1-b'$ then $b\neq 0$, hence by assumption (density) there exists a non-zero  
$a\in \A$ such that $a\leq b$, i.e $a\leq (1- b')$. If $a\cdot s\neq 0$ for some $s\in S$, then $a$ is not less than $b'$ which is impossible.
So $a\cdot s=0$  for every $s\in S$, implying that $a=0$, contradiction.
Now we prove the second part. Assume that $\A\subseteq_d \D$ and $\D$ is atomic. Let $b\in \D$ be an atom. We show that $b\in \At\A$.
By density there is a 
non--zero $a'\in \A$, such that $a'\leq b$ in $\D$. Since $\A$ is atomic, there is an atom $a\in \A$ such that $a\leq a'\leq b$. 
But $b$ is an atom of  $\D$,  and $a$ is non--zero in $\D$, too, so it must be the case that $a=b\in \At\A$.
Thus $\At\B\subseteq \At\A$ and we are done.
\end{proof}

\begin{theorem}\label{iii}
Let $2<n<\omega$. Then any class $\bold K$ such that $\Nr_n\CA_{\omega}\cap {\sf CRCA}_n\subseteq {\bold K}\subseteq \bold S_c\Nr_n\CA_{n+3}$, $\bold K$  
is not elementary. 
\end{theorem}
\begin{proof} The proof is divided into two parts:
 
{\bf 1. Any class between $\bold S_d\Nr_n\CA_{\omega}\cap \CRCA_n$ and $\bold S_c\Nr_n\CA_{n+3}$ is not elementary:}
We use the construction in \cite[Theorem 5.12]{mlq}. 
The algebra $\C_{\Z, \N}(\in \RCA_n$) based on $\Z$ (greens) and $\N$ (reds) denotes the rainbow-like algebra used in {\it op.cit}. 
It was shown in \cite{mlq} that \pe\ has a \ws\ in $G_k(\At\C_{\Z, \N})$ for all $k\in \omega$. With some more effort 
it can be shown that \pe\ has a \ws\ in $\bold H_k(\At\C_{\Z, \N} )$ for all $k\in \omega$
dealing with the new amalgamation moves without hyperedges. 
It remains therefore to describe \pe's strategy in dealing with labelling hyperedges in $\lambda$--neat hypernetworks, where $\lambda$ is a constant label kept on short hyperedges.
In a play, \pe\ is required to play $\lambda$--neat hypernetworks, so she has no choice about the
the short edges, these are labelled by $\lambda$. In response to a cylindrifier move by \pa\
extending the current hypernetwork providing a new node $k$,
and a previously played coloured hypernetwork $M$
all long hyperedges not incident with $k$ necessarily keep the hyperlabel they had in $M$.
All long hyperedges incident with $k$ in $M$
are given unique hyperlabels not occurring as the hyperlabel of any other hyperedge in $M$.
In response to an amalgamation move, which involves two hypernetworks required to be amalgamated, say $(M,N)$
all long hyperedges whose range is contained in $\nodes(M)$
have hyperlabel determined by $M$, and those whose range is contained in $\nodes(N)$ have hyperlabels determined
by $N$. If $\bar{x}$ is a long hyperedge of \pe\ s response $L$ where
$\rng(\bar{x})\nsubseteq \nodes(M)$, $\nodes(N)$ then $\bar{x}$
is given
a new hyperlabel, not used in 
any previously played hypernetwork and not used within $L$ as the label of any hyperedge other than $\bar{x}$.
This completes her strategy for labelling hyperedges.
{\bf We first show that  \pe\ has a \ws\ in $G_k(\At\C_{\Z, \N})$ where $0<k<\omega$ is the number of rounds:}
Let $0<k<\omega$. We proceed inductively. Let $M_0, M_1,\ldots, M_r$, $r<k$ be the coloured graphs at the start of a play of $G_k$ just before round $r+1$.
Assume inductively, that \pe\ computes a partial function $\rho_s:\Z\to \N$, for $s\leq r:$
Let $0<k<\omega$. We proceed inductively. Let $M_0, M_1,\ldots, M_r$, $r<k$ be the coloured graphs at the start of a play of $G_k$ just before round $r+1$.
Assume inductively, that \pe\ computes a partial function $\rho_s:\Z\to \N$, for $s\leq r:$
\begin{enumroman}
\item $\rho_0\subseteq \ldots \rho_t\subseteq\ldots\subseteq\ldots  \rho_s$ is (strict) order preserving; if $i<j\in \dom\rho_s$ then $\rho_s(i)-\rho_s(j)\geq  3^{k-r}$, where $k-r$
is the number of rounds remaining in the game,
and 
$$\dom(\rho_s)=\{i\in \Z: \exists t\leq s, \text { $M_t$ contains an $i$--cone as a subgraph}\},$$

\item for $u,v,x_0\in \nodes(M_s)$, if $M_s(u,v)=\r_{\mu,k}$, $\mu, k\in \N$, $M_s(x_0,u)=\g_0^i$, $M_s(x_0,v)=\g_0^j$,
where $i,j\in \Z$ are tints of two cones, with base $F$ such that $x_0$ is the first element in $F$ under the induced linear order,
then $\rho_s(i)=\mu$ and $\rho_s(j)=k$.
\end{enumroman} 
For the base of the induction \pe\ takes $M_0=\rho_0=\emptyset.$ 
Assume that $M_r$, $r<k$  ($k$ the number of rounds) is the current coloured graph and that \pe\ has constructed $\rho_r:\Z\to \N$ to be a finite order preserving partial map
such conditions (i) and (ii) hold. We show that (i) and (ii) can be maintained in a 
further round.
We check the most difficult case. Assume that $\beta\in \nodes(M_r)$, $\delta\notin \nodes(M_r)$ is chosen by \pa\ in his cylindrifier move,
such that $\beta$ and $\delta$ are apprexes of two cones having
same base and green tints $p\neq  q\in \Z$. 
Now \pe\ adds $q$ to $\dom(\rho_r)$ forming $\rho_{r+1}$ by defining the value $\rho_{r+1}(p)\in \N$ 
in such a way to preserve the (natural) order on $\dom(\rho_r)\cup \{q\}$, that is maintaining property (i).
Inductively, $\rho_r$ is order preserving and `widely spaced' meaning that the gap between its elements is
at least $3^{k-r}$, so this can be maintained in a further round.
Now \pe\  has to define a (complete) coloured graph 
$M_{r+1}$ such that $\nodes(M_{r+1})=\nodes(M_r)\cup \{\delta\}.$ 
In particular, she has to find a suitable 
red label for the edge $(\beta, \delta).$
Having $\rho_{r+1}$ at hand she proceeds as follows. Now that $p, q\in \dom(\rho_{r+1})$, 
she lets $\mu=\rho_{r+1}(p)$, $b=\rho_{r+1}(q)$. The red label she chooses for the edge $(\beta, \delta)$ is: (*)\ \  $M_{r+1}(\beta, \delta)=\r_{\mu,b}$.
This way she maintains property (ii) for $\rho_{r+1}.$  Next we show that this is a \ws\ for \pe. 

We check consistency of newly created triangles proving that $M_{r+1}$ is a coloured graph completing the induction. 
Since $\rho_{r+1}$ is chosen to preserve order, no new forbidden triple (involving two greens and one red) will be created.
Now we check red triangles only of the form $(\beta, y, \delta)$ in $M_{r+1}$ $(y\in \nodes(M_r)$). 
We can assume that  $y$ is the apex of a cone with base $F$ in $M_r$ and green tint $t$, say,
and that $\beta$ is the appex of the $p$--cone having the same base. 
Then inductively by condition (ii), taking $x_0$ to be the first element of $F$, and taking  
the nodes $\beta, y$, and the tints $p, t$, for $u, v, i, j$,  respectively, we have by observing that 
$\beta, y\in \nodes(M_r)$, $\beta, y\in \dom(\rho_r)$ and $\rho_r\subseteq \rho_{r+1}$, 
the following:  
$M_{r+1}(\beta,y)=M_{r}(\beta, y)=\r_{\rho_{r}(p), \rho_{r}(t)}=r_{\rho_{r+1}(p), \rho_{r+1}(t)}.$
By  her strategy, we have  $M_{r+1}(y,\delta)=\r_{\rho_{r+1}(t), \rho_{r+1}(q)}$ 
and we know by (*) that $M_{r+1}(\beta, \delta)=\r_{\rho_{r+1}(p), \rho_{r+1}(q)}$. 
The triple $(\r_{\rho_{r+1}(p), \rho_{r+1}(t)}, \r_{\rho_{r+1}(t), \rho_{r+1}(q)}, \r_{\rho_{r+1}(p), \rho_{r+1}(q)})$
of reds is consistent and we are done with this case. 
All other edge labelling and colouring $n-1$ tuples in $M_{r+1}$ 
by yellow shades are  exactly like in \cite{HH}.  
But we can go further.\\ 
{\bf We show that \pe\ has  a \ws\ in the stronger game $H_k(\At\C)$ for all $k\in \omega$:}.
\pe's strategy dealing with $\lambda$--neat hypernetworks, where $\lambda$ is a constant label kept on short hyperedges is already dealt with.
Now we change the board of play but only formally. We play on {\it $\lambda$--neat hypergraphs}.  
Given a rainbow algebra $\A$, there is a one to one correspondence between coloured graphs on $\At\A$ and networks on $\At\A$ \cite[Half of p. 76]{HHbook2}
denote this correspondence, expressed by a bijection from coloured graphs to networks by  (*): 
$$\Gamma\mapsto N_{\Gamma}, \ \ \nodes(\Gamma)=\nodes(N_{\Gamma}).$$
Now the game $H$ can be re-formulated to be played on {\it $\lambda$--neat hypergraphs} on a rainbow algebra $\A$; these are of the  form
$(\Delta, N^h)$, where $\Delta$ is a coloured graph on $\At\A$, $\lambda$ is a hyperlabel, and 
$N^h$ is as before, $N^h: ^{<\omega}\nodes(\Delta)\to \Lambda$,
such that for $\bar x, \bar y\in\; ^{< \omega}\!\nodes(\Delta)$,
if $\bar x\sim\bar y \Rightarrow N^h(\bar x)=N^h(\bar y).$ Here $\bar{x}\sim \bar{y}$, making the obvious translation,
 is the equivalence relation defined by:  $x\sim y\iff$ $|x|=|y|$ and $N_{\Delta}(x_i, y_i, \bar{z})\leq {\sf d}_{01}$ for all $i<|x|$ and some 
$\bar{z}\in {}^{n-2}\nodes(\Delta)$.
All notions earlier defined for hypernetworks, in particular, $\lambda$--neat ones,  translate to 
$\lambda$--neat hypergraphs, using (*),
like short hyperdges, long hypedges, $\lambda$--neat hypergraphs, etc.
The game is played now on $\lambda$--neat hypergraphs on which the constant label
$\lambda$ is kept on the short hyperedges in $^{<\omega}\nodes(\Delta)$.
We have already dealt with the `graph part' of the game.
We turn to the remaining amalgamation moves. We need some notation and terminology.
Every edge of any hypergraph (edge of its graph part) has an {\it owner \pa\ or \pe}, namely, the one who coloured this edge.
We call such edges \pa\ edges or \pe\ edges. Each long hyperedge $\bar{x}$ in $N^h$ of a hypergraph $N$
occurring in the play has {\it an envelope} $v_N(\bar{x})$ to be defined shortly.\\
In the initial round,  \pa\ plays $a\in \alpha$ and \pe\ plays $N_0$
then all edges of $N_0$ belongs to \pa.
There are no long hyperedges in $N_0$.
If \pa\ plays a cylindrifier move requiring a new node $k$ and \pe\ responds with $M$ then the owner
in $M$ of an edge not incident with $k$ is the same as it was in $N$
and the envelope in $M$ of a long hyperedge not incident with $k$ is the same as that it was in $N$.
All  edges $(l,k)$
for $l\in \nodes(N)\sim \{k\}$ belong to \pe\ in $M$.
if $\bar{x}$ is any long hyperedge of $M$ with $k\in \rng(\bar{x})$, then $v_M(\bar{x})=\nodes(M)$.\\
If \pa\ plays the amalgamation move $(M,N)$ (of two $\lambda$--neat hypergraphs) and \pe\ responds with $L$
then for $m\neq n\in \nodes(L)$ the owner in $L$ of a edge $(m,n)$ is \pa\ if it belongs to
\pa\ in either $M$ or $N$, in all other cases it belongs to \pe\ in $L$.
If $\bar{x}$ is a long hyperedge of $L$
then $v_L(\bar{x})=v_M(\bar{x})$ if $\rng(\bar{x})\subseteq \nodes(M)$, $v_L(\bar{x})=v_N(\bar{x})$ and  $v_L(\bar{x})=\nodes(M)$ otherwise. If in a later move,
\pa\ plays the transformation move $(N,\theta)$
and \pe\ responds with $N\theta$, then owners and envelopes are inherited in the obvious way.
This completes the definition of owners and envelopes.
The next claim, basically, reduces amalgamation moves to cylindrifier moves.
By induction on the number of rounds one can show:

{\bf Claim}:\label{r}  
Let $M, N$ occur in a play of $H_m$, $0<m\in \omega.$ in which \pe\ uses the above labelling
for hyperedges. Let $\bar{x}$ be a long hyperedge of $M$ and let $\bar{y}$ be a long hyperedge of $N$.
Then for any hyperedge $\bar{x}'$ with $\rng(\bar{x}')\subseteq v_M(\bar{x})$, if $M(\bar{x}')=M(\bar{x})$
then $\bar{x}'=\bar{x}$. 
If $\bar{x}$ is a long hyperedge of $M$ and $\bar{y}$ is a long hyperedge of $N$, and $M(\bar{x})=N(\bar{y}),$
then there is a local isomorphism $\theta: v_M(\bar{x})\to v_N(\bar{y})$ such that
$\theta(x_i)=y_i$ for all $i<|x|$. For any $x\in \nodes(M)\sim v_M(\bar{x})$ and $S\subseteq v_M(\bar{x})$, if $(x,s)$ belong to \pa\ in $M$
for all $s\in S$, then $|S|\leq 2$.

Next,  we proceed inductively with the inductive hypothesis exactly as before, except that now each $N_r$ is a
$\lambda$--neat hypergraph.
All what remains is the amalgamation move. With the above claim at hand,
this turns out an easy task to implement guided by \pe\ s
\ws\ in the graph part.\\ 
We consider an amalgamation move at round $0<r$, $(N_s,N_t)$ chosen by 
\pa\ in round $r+1$, \pe\ has to deliver an amalgam $N_{r+1}$.
 \pe\ lets $\nodes(N_{r+1})=\nodes(N_s)\cup \nodes (N_t)$, then she, for a start, 
has to choose a colour for each edge $(i,j)$ where $i\in \nodes(N_s)\sim \nodes(N_t)$ and $j\in \nodes(N_t)\sim \nodes(N_s)$.
Let $\bar{x}$ enumerate $\nodes(N_s)\cap \nodes(N_t).$
If $\bar{x}$ is short, then there are at most two nodes in the intersection
and this case is identical to the cylindrifier move.
If not, that is if $\bar{x}$ is long in $N_s$, then by the claim
there is a partial isomorphism $\theta: v_{N_s}(\bar{x})\to v_{N_t}(\bar{x})$ fixing
$\bar{x}$. We can assume that
$v_{N_s}(\bar{x})=\nodes(N_s)\cap \nodes (N_t)=\rng(\bar{x})=v_{N_t}(\bar{x}).$
It remains to label the edges $(i,j)\in N_{r+1}$ where $i\in \nodes(N_s)\sim \nodes (N_t)$ and $j\in \nodes(N_t)\sim \nodes(N_s)$.
Her strategy is now again similar to the cylindrifier move. If $i$ and $j$ are 
tints of the same cone she chooses a red using $\rho_{r+1}$ (constructed inductively like in the above proof),
if not she  chooses  a white. She never chooses a green.
Concerning $n-1$ tuples 
she  needs to label $n-1$ hyperedges by shades of yellow.
For each tuple $\bar{a}=a_0,\ldots a_{n-2}\in N_{r+1}$,   with no edge
$(a_i, a_j)$ coloured green (we have already labelled edges), then  \pe\ colours $\bar{a}$ by $\y_S$, where
$$S=\{i\in \Z: \text { there is an $i$ cone in $N_{r+1}$ with base $\bar{a}$}\}.$$
We have shown  
that  \pe\ has a \ws\ in $\bold H_{k}(\At\C)$ for each finite $k$.

Using ultrapowers and an elementary chain argument as in \cite[Theorem 3.3.5]{HHbook2}, 
one gets a countable and atomic $\B\in \CA_n$ such \pe\ has a \ws\ in $\bold H_{\omega}(\At(\B)$),
$\B\equiv \C_{\Z, \N} $ and by Lemma \ref{gripneat} $\Cm\At\B\in \Nr_n\CA_{\omega}$ and $\At\B\in \At\Nr_n\CA_{\omega}$ 
(however, we do not guarantee that $\B$ itself is in $\Nr_n\CA_{\omega}$) .  
Since $\B\subseteq_d \Cm\At\B$,
$\B\in \bold S_d\Nr_n\CA_{\omega}$, so $\B\in \bold S_c\Nr_n\CA_{\omega}$.  Being countable, it follows by \cite[Theorem 5.3.6]{Sayedneat} that $\B\in \CRCA_n$. 
We show that 
 \pa\ has a \ws\ in $\bold G^{n+3}(\At\C_{\Z, \N} )$. For the reader's conveniance we include this short part of the proof. 
In the initial round \pa\ plays a graph $M$ with nodes $0,1,\ldots, n-1$ such that $M(i,j)=\w_0$
for $i<j<n-1$
and $M(i, n-1)=\g_i$
$(i=1, \ldots, n-2)$, $M(0, n-1)=\g_0^0$ and $M(0,1,\ldots, n-2)=\y_{\Z}$. This is a $0$ cone.
In the following move \pa\ chooses the base  of the cone $(0,\ldots, n-2)$ and demands a node $n$
with $M_2(i,n)=\g_i$ $(i=1,\ldots, n-2)$, and $M_2(0,n)=\g_0^{-1}.$
\pe\ must choose a label for the edge $(n+1,n)$ of $M_2$. It must be a red atom $r_{mk}$, $m, k\in \N$. Since $-1<0$, then by the `order preserving' condition
we have $m<k$.
In the next move \pa\ plays the face $(0, \ldots, n-2)$ and demands a node $n+1$, with $M_3(i,n)=\g_i$ $(i=1,\ldots, n-2)$,
such that  $M_3(0, n+2)=\g_0^{-2}$.
Then $M_3(n+1,n)$ and $M_3(n+1, n-1)$ both being red, the indices must match.
$M_3(n+1,n)=r_{lk}$ and $M_3(n+1, r-1)=r_{km}$ with $l<m\in \N$.
In the next round \pa\ plays $(0,1,\ldots n-2)$ and re-uses the node $2$ such that $M_4(0,2)=\g_0^{-3}$.
This time we have $M_4(n,n-1)=\r_{jl}$ for some $j<l<m\in \N$.
Continuing in this manner leads to a decreasing
sequence in $\N$. We have proved the required.
 By 
Lemma \ref{n}, $\C_{\Z, \N} \notin \bold S_c\Nr_n\CA_{n+3}$.  
Let $\bold K$ be a class between $\bold S_d\Nr_n\CA_{\omega}\cap {\sf CRCA}_n$ 
and $\bold S_c\Nr_n\CA_{n+3}$. 
Then  $\bold K$ is not elementary, because $\C_{\Z, \N} \notin \bold S_c\Nr_n\CA_{n+3}(\supseteq \bold K)$, 
$\B\in \bold S_d\Nr_n\CA_{\omega}\cap {\sf CRCA}_n(\subseteq \bold K)$,
and $\C_{\Z, \N}\equiv \B$. We are not there yet, for $\B$ might still be outside $\Nr_n\CA_{\omega}$ (like the $\E$ used in item (1) of Theoremref{square}). 


{\bf 2. Proving the required; removing the $\bold S_d$:} To get the required we resort to an auxiliary construction. 
We slighty modify the construction in \cite[Lemma 5.1.3, Theorem 5.1.4]{Sayedneat}. Using the same notation, the algebras $\A$ and $\B$ constructed in {\it op.cit} satisfy 
$\A\in {\sf Nr}_n\CA_{\omega}$, $\B\notin {\sf Nr}_n\CA_{n+1}$ and $\A\equiv \B$.
As they stand, $\A$ and $\B$ are not atomic, but it 
can be  fixed that they are atomic, giving the same result with the rest of the proof unaltered. This is done by interpreting the uncountably many tenary relations in the signature of 
$\Mo$ defined in \cite[Lemma 5.1.3]{Sayedneat}, which is the base of $\A$ and $\B$ 
to be {\it disjoint} in $\Mo$, not just distinct.  The construction is presented this way in \cite{IGPL}, where (the equivalent of) 
$\Mo$ is built in a 
more basic step-by-step fashon.
We work with $2<n<\omega$ instead of only $n=3$. The proof presented in {\it op.cit} lift verbatim to any such $n$.
Let $u\in {}^nn$. Write $\bold 1_u$ for $\chi_u^{\Mo}$ (denoted by $1_u$ (for $n=3$) in \cite[Theorem 5.1.4]{Sayedneat}.) 
We denote by $\A_u$ the Boolean algebra $\Rl_{\bold 1_u}\A=\{x\in \A: x\leq \bold 1_u\}$ 
and similarly  for $\B$, writing $\B_u$ short hand  for the Boolean algebra $\Rl_{\bold 1_u}\B=\{x\in \B: x\leq \bold 1_u\}.$
Then exactly like in \cite{Sayedneat}, it can be proved that $\A\equiv \B$.
Using that $\Mo$ has quantifier elimination we get, using the same argument in {\it op.cit} 
that $\A\in \Nr_n\CA_{\omega}$.  The property that $\B\notin \Nr_n\CA_{n+1}$ is also still maintained.
To see why, consider the substitution operator $_{n}{\sf s}(0, 1)$ (using one spare dimension) as defined in the proof of \cite[Theorem 5.1.4]{Sayedneat}.
Assume for contradiction that 
$\B=\Nr_{n}\C$, with $\C\in \CA_{n+1}.$ Let $u=(1, 0, 2,\ldots n-1)$. Then $\A_u=\B_u$
and so $|\B_u|>\omega$. The term  $_{n}{\sf s}(0, 1)$ acts like a substitution operator corresponding
to the transposition $[0, 1]$; it `swaps' the first two co--ordinates.
Now one can show that $_{n}{\sf s(0,1)}^{\C}\B_u\subseteq \B_{[0,1]\circ u}=\B_{Id},$ 
so $|_{n}{\sf s}(0,1)^{\C}\B_u|$ is countable because $\B_{Id}$ was forced by construction to be 
countable. But $_{n}{\sf s}(0,1)$ is a Boolean automorpism with inverse
$_{n}{\sf s}(1,0)$, 
so that $|\B_u|=|_{n}{\sf s(0,1)}^{\C}\B_u|>\omega$, contradiction.

It can be proved exactly like in \cite{Sayedneat} that the property $\A\equiv \B$ is also still maintained after making the atoms disjoint. In fact, this change offers more for it can be proved 
that $\A\equiv_{\infty,\omega}\B.$
We show that \pe\ has a \ws\ in an \ef-game over $(\A, \B)$ concluding that $\A\equiv_{\infty}\B$.
At any stage of the game,
if \pa\ places a pebble on one of
$\A$ or $\B$, \pe\ must place a matching pebble,  on the other
algebra.  Let $\b a = \la{a_0, a_1, \ldots, a_{n-1}}$ be the position
of the pebbles played so far (by either player) on $\A$ and let $\b
b = \la{b_0, \ldots, b_{n-1}}$ be the the position of the pebbles played
on $\B$.  \pe\ maintains the following properties throughout the game:
For any atom $x$ (of either algebra) with
$x\cdot \bold 1_{Id}=0$ then $x \in a_i\iff x\in b_i$ and $\b a$ induces a finite partion of $\bold 1_{Id}$ in $\A$ of $2^n$
(possibly empty) parts $p_i:i<2^n$ and $\b b$ induces a partion of
$\bold 1_{Id}$ in $\B$ of parts $q_i:i<2^n$.  
Furthermore, $p_i$ is finite $\iff$ $q_i$ is
finite and, in this case, $|p_i|=|q_i|$.
That such properties can be maintained is fairly easy to show.
Now because $\A\in \Nr_n\CA_{\omega}\cap {\sf CRCA}_n$, it suffices to show  
(since $\B$ is atomic) that 
$\B$  is in fact outside $\bold S_d\Nr_n\CA_{n+1}\cap \bf At$.  
Take $\kappa$ the signature of $\Mo$; more specifically,  the number of $n$-ary relation symbols to be $2^{2^{\omega}}$, and assume for contradiction that  
$\B\in \bold S_d\Nr_n\CA_{n+1}\cap \bf At$. 
Then $\B\subseteq_d \mathfrak{Nr}_n\D$, for some $\D\in \CA_{n+1}$ and $\mathfrak{Nr}_n\D$ is atomic. For brevity, 
let $\C=\mathfrak{Nr}_n\D$. Then by item (1) of Lemma \ref{join} $\Rl_{Id}\B\subseteq_d \Rl_{Id}\C$.
Since $\C$ is atomic,  then by item (1) of the same Lemma $\Rl_{Id}\C$ is also atomic.  Using the same reasoning as above, we get that $|\Rl_{Id}\C|>2^{\omega}$ (since $\C\in \Nr_n\CA_{n+1}$.) 
By the choice of $\kappa$, we get that $|\At\Rl_{Id}\C|>\omega$. 
By density, we get from item (2) of Lemma \ref{join}, that $\At\Rl_{Id}\C\subseteq \At\Rl_{Id}\B$. 
Hence $|\At\Rl_{Id}\B|\geq |\At\Rl_{Id}\C|>\omega$.   
But by the construction of $\B$, $|\Rl_{Id}\B|=|\At\Rl_{Id}\B|=\omega$,   which is a  contradiction and we are done.
Since ${\bf El}(\Nr_n\CA_{\omega}\cap {{\sf CRCA}_n})\nsubseteq \bold S_d\Nr_n\CA_{\omega}\cap \CRCA_n$,
there can be no elementary class  
between $\Nr_n\CA_{\omega}\cap {{\sf CRCA}_n}$ and $\bold S_d\Nr_n\CA_{\omega}\cap \CRCA_n$. Having already 
eliminated elementary classes between $\bold S_d\Nr_n\CA_{\omega}\cap {\sf CRCA}_n$ and $\bold S_c\Nr_n\CA_{n+3}$, we are done.

\end{proof}

Fix finite $k>2$.  Then $\V_k={\sf Str}(\bold S\Nr_n\CA_{n+k})$ is not elementary 
$\implies   \V_k$ is not-atom canonical.  But the converse implication {\it does not hold} because (argueing contrapositively) in the case of atom--canonicity, 
we get that ${\sf Str}(\bold S\Nr_n\CA_{n+k})={\sf At}(\bold S\Nr_n\CA_{n+k})$, and the last class is elementary \cite[Theorem 2.84]{HHbook}.  
In particular, we do not know whether ${\sf Str}(\bold S\Nr_n\CA_{n+k})$, for a particular finite $k\geq 3$, is elementary or not.
Nevertheless, it is easy to show that {\it there has to be a finite $k<\omega$ such that $\V_j$ is not elementary for all $j\geq k$}:
\begin{theorem}\label{truncate} \begin{enumerate} 
\item There is a finite $k\geq 2$, such that for all $m\geq n+k$ the class of frames ${\sf Str}(\bold S\Nr_n\CA_{m})=\{\F: \Cm\F\in \bold S\Nr_n\CA_{m}\}$ is not elementary. 
An entirely analogous result holds for $\RA$s,
\item  Let $\bold O\in \{\bold S_c, \bold S_d, \bold I\}$ and $k\geq 3$. 
Then the class of frames ${\sf K}_k=\{\F: \Cm\F\in \bold O\Nr_n\CA_{n+k}\}$ 
is not elementary. 

\end{enumerate}
\end{theorem}
\begin{proof}
(1): 
We show that ${\sf Str}(\bold S\Nr_n\CA_{m})$ is not elementary for some finite $m\geq n+2$. By \cite{a} $m$ cannot be equal to $n+1$. 
Let $(\A_i: i\in \omega)$ be a sequence of (strongly) representable $\CA_n$s with $\Cm\At\A_i=\A_i$
and $\A=\Pi_{i/U}\A_i$ is not strongly representable with respect to any non-principal ultrafilter $U$ on $\omega$.
Such algebras exist \cite{HHbook2}. Hence $\Cm\At\A\notin \bold S\Nr_n\CA_{\omega}=\bigcap_{i\in \omega}\bold S\Nr_n\CA_{n+i}$, 
so $\Cm\At\A\notin \bold S\Nr_n\CA_{l}$ for all $l>m$, for some $m\in \omega$, $m\geq n+2$. 
But for each such $l$, $\A_i\in \bold S\Nr_n\CA_l(\subseteq {\sf RCA}_n)$, 
so $(\A_i:i\in \omega)$ is a sequence of algebras such that $\Cm\At(\A_i)\in \bold S\Nr_n\CA_{l}$ $(i\in I)$, but 
$\Cm(\At(\Pi_{i/U}\A_i))=\Cm\At(\A)\notin \bold S\Nr_n\CA_l$, for all $l\geq m$.\\
(2): We use the same construction (and notation) in the last item of Theorem \ref{iii}. 
It suffices to show that the class of algebras $\bold K_k=\{\A\in \CA_n\cap {\bf At}: \Cm\At\A\in \bold O\Nr_n\CA_k\}$ is not elementary. 
\pe\ has a \ws\ in $\bold H_{\omega}(\alpha)$ for some countable atom structure $\alpha$, 
$\Tm\alpha\subseteq_d \Cm\alpha\in \bold {\sf Nr}_n\CA_{\omega}$ and $\Tm\alpha\in {\sf CRCA}_n$.
Since $\C_{\Z, \N}\notin \bold S_c{\sf Nr}_n\CA_{n+3}$,  then 
$\C_{\Z, \N}=\Cm\At\C_{\Z, \N}\notin {\bold K}_k$, $\C_{\Z, \N}\equiv \Tm\alpha$ 
and $\Tm\alpha\in {\bold K}_k$ 
because $\Cm\alpha\in \Nr_n\CA_{\omega}\subseteq \bold S_d\Nr_n\CA_{\omega}\subseteq \bold S_c\Nr_n\CA_{\omega}$.
We have shown that $\C_{\Z, \N}\in {\bf El}{\bold K}_k\sim {\bold K}_k$, proving the required. 


\end{proof}

To obtain the $\sf RA$ analogue of item (2) of  Theorem \ref{truncate}, we need to strengthen \cite[Theorem 39]{r}. 
We prove more by allowing infinite conjunctions in constructing a certain 
model (denoted by ${\cal M}_a$) as clarified below. 
The $k$ rounded game $H_k$ $(k\leq \omega)$ is defined for relation algebras in \cite[Definition 28]{r}. 
$\sf Gws_{\beta}$ denotes the class of {\it generalized weak set algebras of dimension $\beta$} in the sense of \cite[Definition 3.1.2]{HMT2}.
\begin{theorem}\label{gripneat} Let $\alpha$ be a countable atom structure. If \pe\ has a \ws\ in $H_{\omega}(\alpha)$, then
there exists a complete $\D\in \RCA_{\omega}$ such that 
$\Cm\alpha\cong \mathfrak{Ra}\D$ and $\alpha\cong \At\mathfrak{Ra}\D$. In particular, $\Cm\alpha\in \Ra\CA_{\omega}$, 
$\alpha\in \At\Ra\CA_{\omega}$ and $\alpha$ is completely representable.  
\end{theorem} 
\begin{proof}
Fix some $a\in\alpha$. As shown in \cite{r}, the game $H_{\omega}$ is designed so that using \pe\ s \ws\ in the game $H_{\omega}(\alpha)$ 
one can define a
nested sequence $M_0\subseteq M_1,\ldots$ of $\lambda$--neat networks
where $M_0$ is \pe's response to the initial \pa-move $a$ such that: 
If $M_r$ is in the sequence and $M_r(x,y)\leq a;b$ for an atoms $a$ and $b$
then there is $s\geq r$ and a witness $z\in\nodes(M_s)$
such that  $M_s(x,z)=a$ and $M_s(z, y)=b$.
In addition, if $M_r$ is in the sequence and $\theta$ is any partial
isomorphism of $M_r$, then there is $s\geq r$ and a
partial isomorphism $\theta^+$ of $M_s$ extending $\theta$ such that
$\rng(\theta^+)\supseteq\nodes(M_r)$. 
Now let $M_a$ be the limit of this sequence as defined in \cite{r}.
Let $L$ be the signature with one binary relation for
each $b\in\alpha$, and one $k$--ary predicate symbol for
each $k$--ary hyperlabel $\lambda$.
{\it We work in $L_{\infty, \omega}.$}
For fixed $f_a\in\;^\omega\!\nodes(M_a)$, let
$\U_a=\set{f\in\;^\omega\!\nodes(M_a):\set{i<\omega:g(i)\neq
f_a(i)}\mbox{ is finite}}$.
One makes $\U_a$ into the base of an $L$ relativized structure 
${M}_a$ like in \cite[Theorem 29]{r} except that we allow a clause for infinitary disjunctions.
We are now
working with (weak) set algebras  whose semantics are induced by $L_{\infty, \omega}$ formulas in the signature $L$,
instead of first order ones.
For any such $L$-formula $\phi$, write $\phi^{{M}_a}$ for
$\set{f\in\U_a: {M}_a, f\models\phi}.$
Let
$D_a= \set{\phi^{{M}_a}:\phi\mbox{ is an $L$-formula}}$ and
$\D_a$ be the weak set algebra with universe $D_a$. 
Let $\D=\bold P_{a\in \alpha} \D_a$. Then $\D\in \sf Gws_{\omega}$ and furthermore $\D$ is complete.
Suprema exists in $\D$  because we chose to work with $L_{\infty, \omega}$ while forming the dilations $\D_a$ $(a\in \alpha)$. 
Each $\D_a$ is complete, hence so is 
their product $\D$.
Now $\alpha\cong \At\mathfrak{Ra}\D$ as proved in \cite[Theorem 39]{r}. We show that $\Cm\alpha\cong \mathfrak{Ra}\D$.
Since $\D$ is complete,  then 
$\mathfrak{Ra}\D$ is complete.
Making the legitamite identification  $\mathfrak{Ra}\D\subseteq_d \Cm\alpha$, by density
we get that  $\mathfrak{Ra}\D=\Cm\alpha$ because $\mathfrak{Ra}\D$ is complete,  
so $\Cm\alpha\in \Ra\CA_{\omega}$. 
Now $\Cm\alpha\in \bold S_d\Ra\CA_{\omega}(\subseteq \bold S_c\Ra\CA_{\omega}$)
and $\alpha$ is countable, so $\alpha$ is completely representable. 
\end{proof}
\begin{corollary}\label{last}  Let $\bold O\in \{\bold S_c, \bold S_d, \bold I\}$ and $m\geq 5$. 
Then the class of frames ${\sf L}_m=\{\F: \Cm\F\in \bold O\Ra\CA_m\}$ is not elementary.  Furthermore, any class $\sf K$ such that 
$\bold S_d\Ra\CA_{\omega}\cap {\sf CRRA}\subseteq {\sf  K}\subseteq \bold S_c\Ra\CA_{6}$, $\sf K$ is not elementary.
\end{corollary}
\begin{proof} Let $\bold L_m=\{\A\in \RA_n\cap {\bf At}: \Cm\At\A\in \bold O\Ra\CA_m\}$. 
Take the relation algebra atom structure $\beta$ based on $\N$ and $\Z$ as defined in \cite{r}, 
for  which \pe\ has a \ws\ in $H_k(\At\A_{\Z, \N})$, where $\A_{\Z, \N}=\Cm\beta$ for all $k<\omega$, and \pa\ has a \ws\ in $F^5(\beta)$ with $F^5$ as in \cite[Definition 28]{r}.
By \cite[Lemma 2.6]{r}, we get that $\A_{\Z, \N}\notin \bold S_c\Ra\CA_5$. 
The usual argument of taking ultrapowers followed by an elementary chain argument, one gets 
a countable atom structure $\alpha$, such that $\Tm\alpha\equiv \A_{\Z, \N}$ and  
\pe\ has a \ws\ in $H(\alpha)$.
By Theorem \ref{gripneat}, there is 
$\D\in \CA_{\omega}$, such that $\alpha\equiv \beta$, $\alpha\cong \At\D$ and $\Tm\alpha\subseteq_d \mathfrak{Ra}\D$, so that  
$\Tm\alpha\in \bold S_d\Ra\CA_{\omega}$. 
Since $\A_{\Z, \N}\notin \bold S_c\Ra\CA_{5}$,  then 
$\A_{\Z, \N}=\Cm\At\C_{\Z, \N}\notin {\bold L}_m$, $\A_{\Z, \N}\equiv \Tm\alpha$ 
and $\Tm\alpha\in {\bold L}_m$ 
because $\Cm\alpha\in \Ra\CA_{\omega}$. For the second required one 
uses the same elementary equivalent algebras $\Tm\alpha\in \bold S_d\Ra\CA_{\omega}$ and $\A_{\Z,\N}\notin \bold S_c\Ra\CA_5$. 
\end{proof}
In the next table we summarize the results hitherto obtained on first order definability: 
\vskip1mm
\begin{tabular}{|l|c|c|c|c|c|c|}    \hline
					Algebras                    &Elementary        &Citation                              \\

                                                               \hline
                                                                          $\Nr_n\CA_{\omega}\cap {\sf CRCA}_n\subseteq \bold K\subseteq \bold S_d\Nr_n\CA_{n+1}$&no& Last item of Thm \ref{iii} \\

                                                               \hline
                                                                          $\Nr_n\CA_{\omega}\cap {\sf CRCA}_n\subseteq \bold K\subseteq \bold S_c\Nr_n\CA_{n+3}$&no& Last item  of Thm \ref{iii}\\

                                                                           \hline
                                                                          $\At(\bold \Nr_n\CA_{\omega}\cap {\sf CRCA}_n)\subseteq \bold K\subseteq \bold \At \bold S_c\Nr_n\CA_{n+3}$  &no& From result in previous row\\

                                                                           \hline
                                                                          $\Nr_n\CA_{\omega}\subseteq \bold K\subseteq \bold \Nr_n\CA_{n+1}$   &no &\cite{IGPL},  \cite[Theorem 5.4.2]{Sayedneat}\\

                                                               \hline
                                                                          $\bold S_c\Ra\CA_{\omega}\cap {\sf CRRA}\subseteq \bold K\subseteq \bold S_c\Ra\CA_{6}$&no& \cite{r2}\\

                                                               \hline
                                                                          $\bold S_d\Ra\CA_{\omega}\cap {\sf CRRA}\subseteq \bold K\subseteq \bold S_c\Ra\CA_{6}$&no& Corollary \ref{last} \\

                                                               \hline
                                                                          $\Ra\CA_{\omega}\cap {\sf CRRA}_n\subseteq \bold K\subseteq \bold S_c\Ra\CA_{6}$&?& \\

                                                                           \hline
                                                                          $\At(\Ra\CA_{\omega}\cap {\sf CRRA})\subseteq \bold K\subseteq \At \bold S_c\Ra\CA_{6}$  &no&  Corollary \ref{last}, cf.\cite{r}\\

\hline

\end{tabular}
\vskip2mm
The result in row seven is stronger than the result proved in \cite{r2} because 
$\bold S_d\leq \bold S_c$.
Neither of the results in the second and third row is superfluous 
for the two classes $\bold S_c\Nr_n\CA_{n+3}$ and $\bold S_d\Nr_n\CA_{n+1}$ are mutually distinct: 
The algebra $\B$ in the last item of Theorem \ref{iii} is completely representable, so 
$\B\in \bold S_c\Nr_n\CA_{n+3}\sim \bold S_d\Nr_n\CA_{n+1}$.
Conversely, in \cite[\S2]{t}, a finite algebra $\A\in \Nr_n\CA_{n+1}\sim \bold S\Nr_n\CA_{n+2}(\subseteq \bold S_d\Nr_n\CA_{n+1}\sim \bold S_c\Nr_n\CA_{n+3}$) is constructed. 
From the no's in the sixth and last row we cannot answer (negatively) the only question mark in the table. This would be a far too hasty decision (taken mistakenly in \cite{r} and corrected in \cite{r2}). 
The implication $\At\R\in \At\Ra\CA_{\omega}$ and $\Cm\At\R\in \Ra\CA_{\omega}\implies \R\in \Ra\CA_{\omega}$ may not be valid.  
Indeed, the algebra $\B$ used in item (1) of Theorem \ref{square}
confirms our doubts, 
since $\At\B\in \At\Nr_n\CA_{\omega}$, $\Cm\At\B\in \Nr_n\CA_{\omega}$, but $\B\notin \Nr_n\CA_{\omega}$. 
 
Unless otherwise explicitly indicated, fix $2<n<\omega$:
The last example motivates:
\begin{definition}\label{grip} 
\begin{enumerate}
\item A $\CA_n$ atom structure $\bf At$ is {\it weakly representable} if there is an $\A\in \RCA_n$ such that $\At\A=\bf At$.  
The $\CA_n$ atom structure $\bf At$ is {\it strongly representable} if every $\A\in \RCA_n$ such that $\At\A={\bf At}\implies \A\in \RCA_n$. 
\item The  class $\bold K$ is {\it gripped by its atom structures} or simply {\it gripped,}
if for $\A\in \CA_n$, whenever $\At\A\in \At \bold K$, then $\A\in \bold K$.
An $\omega$--rounded game $\bold H$ {\it grips} $\bold K$, if whenever  $\A\in\CA_n$ is atomic with countably many atoms and \pe\ has a \ws\ in $\bold H(\A)$,
then $\A\in \bold K$. The game $\bold H$ {\it weakly grips} $\bold K$, if whenever  $\A\in\CA_n$ is atomic with countably many atoms and
\pe\ has a \ws\ in $\bold H(\At\A)$, then $\At\A\in \At\bold K$. The game $\bold H$ {\it densely grips} $\bold K$,  if whenever  $\A\in\CA_n$ is atomic with countably many atoms and
\pe\ has a \ws\ in $\bold H(\At\A)$, then $\At\A\in \At\bold K$ and $\Cm\At\A\in \bold K$.

\end{enumerate}
\end{definition}
Let $2<n<m\leq \omega$. It is is known that a $\CA_n$ atom structure $\bf At$ is weakly representable $\iff$ $\Tm\bf At$ (the term algebra which is the subalgebra of the complex algebra 
$\Cm\bf At$ generated by the atoms)
is representable; 
it is strongly representable $\iff$ $\Cm\bf At$ is representable. 
\cite{mlq, Hodkinson}. 
The classes $\RCA_n$ and $\Nr_n\CA_m$ are not gripped, by \cite{Hodkinson} and first item of Theorem \ref{square}. 
In \cite{Hodkinson} a weakly representable $\CA_n$ atom structure that is not strongly representable is constructed, showing that $\RCA_n$ is 
not closed under \de\ completions, because $\Cm\bf At$ is the \de\ completion of 
$\Tm\bf At$ in the case of completely additive varieties of $\sf BAO$s. But we can go even further:   

\begin{theorem} For $m\geq m+3$, the variety $\bold S\Nr_n\CA_m$ is not gripped.
\end{theorem}
\begin{proof} It clearly suffices to  construct an atom structure $\bf At$ such that $\Tm{\bf At}\in \RCA_n$ but $\Cm{\bf At}\notin \bold S\Nr_n\CA_{n+3}$.
We divide the proof into four parts:

1: {\bf Blowing up and blurring  $\A_{n+1, n}$ forming a weakly representable atom structure $\bf At$}:
Take the finite rainbow ${\sf CA}_n$,  $\A_{n+1, n}$
where the reds $\sf R$ is the complete irreflexive graph $n$, and the greens
are  $\{\g_i:1\leq i<n-1\}
\cup \{\g_0^{i}: 1\leq i\leq n+1\}$, endowed with the quasi--polyadic equality operations.
Denote the finite atom structure of $\A_{n+1, n}$ by ${\bf At}_f$; 
so that ${\bf At}_f=\At(\A_{n+1, n})$.
One  then replaces the red colours 
of the finite rainbow algebra of $\A_{n+1, n}$ each by  infinitely many reds (getting their superscripts from $\omega$), obtaining this way a weakly representable atom structure $\bf At$.
The resulting atom structure after `splitting the reds', namely, $\bf At$,  is 
like the weakly (but not strongly) representable 
atom structure of the atomic, countable and simple algebra $\A$ as defined in \cite[Definition 4.1]{Hodkinson}; the sole difference is that we have $n+1$ greens
and not $\omega$--many as is the case in \cite{Hodkinson}. 
We denote the resulting term cylindric algebra $\Tm\bf At$ by ${{\mathfrak{Bb}}}(\A_{n+1, n}, \r, \omega)$ short hand for  blowing 
up and blurring $\A_{n+1,n}$ by splitting each {\it red graph (atom)} into $\omega$ many.  
It can be shown exactly like in \cite{Hodkinson} that \pe\ can win the rainbow $\omega$--rounded game
and build an $n$--homogeneous model $\Mo$ by using a shade of red $\rho$ {\it outside} the rainbow signature, when
she is forced a red;  \cite[Proposition 2.6, Lemma 2.7]{Hodkinson}. The $n$-homogeniuty  entails that any subgraph (substructure)
of $\Mo$ of size $\leq n$, is independent of its location in $\Mo$;
it is uniquely determined by its isomorphism type.
One proves like in {\it op.cit} 
that  ${{\mathfrak{Bb}}}(\A_{n+1, n}, \r, \omega)$ is representable as a set algebra having top element $^n\Mo$. 
We give more details. In the present context, after the splitting `the finitely many red colours' replacing each such red colour $\r_{kl}$, $k<l<n$  by $\omega$ many 
$\r_{kl}^i$, $i\in \omega$, the rainbow signature for the resulting rainbow theory as defined in \cite[Definition 3.6.9]{HHbook} call this theory $T_{ra}$,
 consists of $\g_i: 1\leq i<n-1$, $\g_0^i: 1\leq i\leq n+1$,
$\w_i: i<n-1$,  $\r_{kl}^t: k<l< n$, $t\in \omega$,
binary relations, and $n-1$ ary relations $\y_S$, $S\subseteq_{\omega} n+k-2$ or $S=n+1$. 
The set algebra ${\mathfrak{Bb}}(\A_{n+1, n}, \r, \omega)$ of dimension $n$ has
base an $n$--homogeneous  model $\Mo$ of another theory $T$ whose signature expands that of 
$T_{ra}$ by an additional binary relation (a shade of red) $\rho$.  
In this new signature $T$ is obtained from $T_{ra}$ by 
some axioms  (consistency conditions) extending $T_{ra}$. Such axioms (consistency conditions) 
specify consistent triples involving $\rho$. We call the models of $T$ {\it extended} coloured graphs. 
In particular, $\Mo$ is an extended coloured graph.
To build $\Mo$, the class of coloured graphs is considered in
the signature $L\cup \{\rho\}$ like in uual rainbow constructions as given above with the two additional forbidden triples
$(\r, \rho, \rho)$ and $(\r, \r^*, \rho)$, where $\r, \r^*$ are any reds. 
This model $\Mo$ is constructed as a countable limit of finite models of $T$ 
using a game played between \pe\ and \pa.   Here, unlike the extended $L_{\omega_1, \omega}$ theory 
dealt with in \cite{Hodkinson},  $T$ is a {\it first order one}
because the number of greens used are finite.
In the rainbow game \cite{HH}  \S 4.3,
\pa\ challenges \pe\  with  {\it cones} having  green {\it tints $(\g_0^i)$}, 
and \pe\ wins if she can respond to such moves. This is the only way that \pa\ can force a win.  \pe\ 
has to respond by labelling {\it appexes} of two succesive cones, having the {\it same base} played by \pa.
By the rules of the game, she has to use a red label. She resorts to  $\rho$ whenever
she is forced a red while using the rainbow reds will lead to an inconsitent triangle of reds;  \cite[Proposition 2.6, Lemma 2.7]{Hodkinson}. 
The indicative term 'blow up and blur' was introduced in \cite{ANT}. The 'blowing up' refers to `splitting each red atom of $\A_{n+1, n}$ into infinitely many subatoms called its copies 
obtaining the new weakly represenatle atom structure $\b At$. The term `blur' refers to the
fact that the algebraic structure of $\A_{n+1, n}$ is `disorganized' in $\Tm\bf At$; it is completely distorted by redefining the operations to the point that the original opertions of $\A_{n+1, n}$ 
is blurred; it is no longer there. 
Nevertheless, the algebraic structure of $\A_{n+1, n}$ reappears in $\Cm\bf At$, the \de\ conpletion of $\Tm\bf At$, in the sense 
that $\A_{n+1, n}$ embeds ino $\Cm\bf At$ as will be shown in a while. The completion (supremma exists) of $\Cm\bf At$ plays a key role, 
because the embedding is done by mapping every atom to the suprema of its copies. This suprema may, and indeed in some of the cases does not exist in
$\Tm\bf At$.

2. {\bf Representing $\Tm\At\A$ (and its completion) as (generalized) set algebras:} From now on, forget about $\rho$; 
having done its task as a colour  to (weakly) represent $\bf At$, it will play no further role.
Having $\Mo$ at hand, one constructs  two atomic $n$--dimensional set algebras based on $\Mo$, sharing the same atom structure and having 
the same top element.  
The atoms of each will be the set of coloured graphs, seeing as how, quoting Hodkinson \cite{Hodkinson} such coloured graphs are `literally indivisible'. 
Now $L_n$ and $L_{\infty, \omega}^n$ are taken in the rainbow signature (without $\rho$). Continuing like in {\it op.cit}, deleting the one available red shade, set
$W = \{ \bar{a} \in {}^n\Mo : \Mo \models ( \bigwedge_{i < j <n} \neg \rho(x_i, x_j))(\bar{a}) \},$
and for $\phi\in L_{\infty, \omega}^n$, let
$\phi^W=\{s\in W: \Mo\models \phi[s]\}.$
Here $W$ is the set of all $n$--ary assignments in
$^n\Mo$, that have no edge labelled by $\rho$.
Let $\A$  be the relativized set algebra with domain
$\{\varphi^W : \varphi \,\ \textrm {a first-order} \;\ L_n-
\textrm{formula} \}$  and unit $W$, endowed with the
usual concrete quasipolyadic operations read off the connectives.
Classical semantics for {\it $L_n$ rainbow formulas} and their
semantics by relativizing to $W$ coincide \cite[Proposition 3.13]{Hodkinson} {\it but not with respect to 
$L_{\infty,\omega}^n$ rainbow formulas}.
This depends essentially on \cite[Lemma 3.10]{Hodkinson}, which is the heart and soul of the proof in \cite{Hodkinson}, and for what matters this proof.
The referred to lemma says that any permutation $\chi$ of $\omega\cup \{\rho\}$,
$\Theta^{\chi}$ as defined in  \cite[Definitions 3.9, 3.10]{Hodkinson} is an $n$ back--and--forth system
induced by any permutation of $\omega\cup \{\rho\}$. 
Let $\chi$ be such a permutation. The system $\Theta^{\chi}$ consists of isomorphisms between coloured graphs such that 
superscripts of reds are `re-shuffled along $\chi$', in such a way that rainbow red labels are permuted,  
$\rho$ is replaced by a red rainbow label, and all other colours are preserved.
One uses such $n$-back-and-forth systems mapping a 
tuple $\bar{b} \in {}^n \Mo \backslash W$ to a tuple
$\bar{c} \in W$ preserving any formula in $L_{ra}$ containing the non-red symbols that are
`moved' by the system, so if $\bar{b}\in {}^n\Mo$ refutes the $L_n$ rainbow formula  $\phi$, then there is a $\bar{c}$ in $W$ 
refuting $\phi$. 
Hence the set algebra $\A$ is isomorphic to a quasi-polyadic equality  set algebra of dimension $n$ 
having top element $^n\Mo$, so $\A$
is simple, in fact its $\Df$ reduct is simple.
Let $\E=\{\phi^W: \phi\in L_{\infty, \omega}^n\}$
\cite[Definition 4.1]{Hodkinson}
with the operations defined like on $\A$ the usual way. $\Cm\bf At$ is a complete $\QEA_n$ and, so like in \cite[Lemma 5.3]{Hodkinson}
we have an isomorphism from $\Cm\bf At$  to $\E$ defined
via $X\mapsto \bigcup X$.
Since $\At\A=\At\Tm(\At\A)=\bf At$  
and $\Tm\At\A\subseteq \A$, hence $\Tm\At\A$ is representable.
The atoms of $\A$, $\Tm\At\A$ and $\Cm\At\A=\Cm \bf At$ are the coloured graphs whose edges are {\it not labelled} by $\rho$.
These atoms are uniquely determined by the interpretion in $\Mo$ of so-called $\sf MCA$ formulas in the rainbow signature of $\bf At$  as in
\cite[Definition 4.3]{Hodkinson}.

3. {\bf Embedding $\A_{n+1, n}$ into $\Cm(\At({{\mathfrak{Bb}}}(\A_{n+1, n}, \r, \omega)))$:} Let ${\sf CRG}_f$ be the class of coloured graphs on 
${\bf At}_f$ and $\sf CRG$ be the class of coloured graph on $\bf At$. We 
can (and will) assume that  ${\sf CRG}_f\subseteq \sf CRG$.
Write $M_a$ for the atom that is the (equivalence class of the) surjection $a:n\to M$, $M\in \sf CGR$.
Here we identify $a$ with $[a]$; no harm will ensue.
We define the (equivalence) relation $\sim$ on $\bf At$ by
$M_b\sim N_a$, $(M, N\in {\sf CGR}):$
\begin{itemize}
\item $a(i)=a(j)\Longleftrightarrow b(i)=b(j),$

\item $M_a(a(i), a(j))=\r^l\iff N_b(b(i), b(j))=\r^k,  \text { for some $l,k$}\in \omega,$

\item $M_a(a(i), a(j))=N_b(b(i), b(j))$, if they are not red,

\item $M_a(a(k_0),\dots, a(k_{n-2}))=N_b(b(k_0),\ldots, b(k_{n-2}))$, whenever
defined.
\end{itemize}
We say that $M_a$ is a {\it copy of $N_b$} if $M_a\sim N_b$ (by symmetry $N_b$ is a copy of $M_a$.) 
Indeed, the relation `copy of' is an equivalence relation on $\bf At$.  An atom $M_a$ is called a {\it red atom}, if $M_a$ has at least one red edge. 
Any red atom has $\omega$ many copies, that are {\it cylindrically equivalent}, in the sense that, if $N_a\sim M_b$ with one (equivalently both) red,
with $a:n\to N$ and  $b:n\to M$, then we can assume that $\nodes(N) =\nodes(M)$ 
and that for all $i<n$, $a\upharpoonright n\sim\{i\}=b\upharpoonright n\sim \{i\}$.
Any red atom has $\omega$ many copies that are {\it cylindrically equivalent}, in the sense that, if $N_a\sim M_b$ with one (equivalently both) red,
with $a:n\to N$ and  $b:n\to M$, then we can assume that $\nodes(N) =\nodes(M)$ 
and that for all $i<n$, $a\upharpoonright n\sim\{i\}=b\upharpoonright n\sim \{i\}$.
In $\Cm\bf At$, we write $M_a$ for $\{M_a\}$ 
and we denote suprema taken in $\Cm\bf At$, possibly finite, by $\sum$.
Define the map $\Theta$ from $\A_{n+1, n}=\Cm{\bf At}_f$ to $\Cm\bf At$,
by specifing first its values on ${\bf At}_f$,
via $M_a\mapsto \sum_jM_a^{(j)}$ where $M_a^{(j)}$ is a copy of $M_a$. 
So each atom maps to the suprema of its  copies.  
This map is well-defined because $\Cm\bf At$ is complete. 
We check that $\Theta$ is an injective homomorphim. Injectivity follows from $M_a\leq \Theta(M_a)$, hence $\Theta(x)\neq 0$ 
for every atom $x\in \At(\A_{n+1, n})$.
We check preservation of all the $\QEA_n$ operations.  
The Boolean join is obvious.
\begin{itemize}
\item For complementation: It suffices to check preservation of  complementation `at atoms' of ${\bf At}_f$. 
So let $M_a\in {\bf At}_f$ with $a:n\to M$, $M\in {\sf CGR}_f\subseteq \sf CGR$. Then: 
$$\Theta(\sim M_a)=\Theta(\bigcup_{[b]\neq [a]} M_b)
=\bigcup_{[b]\neq [a]} \Theta(M_b)
=\bigcup_{[b]\neq [a]}\sum_j M_b^{(j)}$$
$$=\bigcup_{[b]\neq [a]}\sim \sum_j[\sim (M_a)^{(j)}]
=\bigcup_{[b]\neq [a]}\sim \sum_j[(\sim M_b)^j]
=\bigcup_{[b]\neq [a]}\bigwedge_j M_b^{(j)}$$
$$=\bigwedge_j\bigcup_{[b]\neq [a]}M_b^{(j)}
=\bigwedge_j(\sim M_a)^{j}
=\sim (\sum M_a^j)
=\sim \Theta(a)$$

\item Diagonal elements. Let $l<k<n$. Then:
\begin{align*}
M_x\leq \Theta({\sf d}_{lk}^{\Cm{\bf At}_f})&\iff\ M_x\leq \sum_j\bigcup_{a_l=a_k}M_a^{(j)}\\
&\iff M_x\leq \bigcup_{a_l=a_k}\sum_j M_a^{(j)}\\
&\iff  M_x=M_a^{(j)}  \text { for some $a: n\to M$ such that $a(l)=a(k)$}\\
&\iff M_x\in {\sf d}_{lk}^{\Cm\bf At}.
\end{align*}

\item Cylindrifiers. Let $i<n$. By additivity of cylindrifiers, we restrict our attention to atoms 
$M_a\in {\bf At}_f$ with $a:n\to M$, and $M\in {\sf CRG}_f\subseteq \sf CRG$. Then: 

$$\Theta({\sf c}_i^{\Cm{\bf At}_f}M_a)=f (\bigcup_{[c]\equiv_i[a]} M_c)
=\bigcup_{[c]\equiv_i [a]}\Theta(M_c)$$
$$=\bigcup_{[c]\equiv_i [a]}\sum_j M_c^{(j)}=\sum_j \bigcup_{[c]\equiv_i [a]}M_c^{(j)}
=\sum _j{\sf c}_i^{\Cm\bf At}M_a^{(j)}$$
$$={\sf c}_i^{\Cm\bf At}(\sum_j M_a^{(j)})
={\sf c}_i^{\Cm\bf At}\Theta(M_a).$$

\end{itemize}

4.: {\bf \pa\ has  a  \ws\ in $G^{n+3}\At(\A_{n+1, n})$; and the required result:} It is straightforward to show that 
\pa\ has \ws\ first in the  \ef\ forth  private 
game played between \pe\ and \pa\ on the complete
irreflexive graphs $n+1$ and $n$ in 
$n+1$ rounds
${\sf EF}_{n+1}^{n+1}(n+1, n)$ \cite [Definition 16.2]{HHbook2}
since $n+1$ is `longer' than $n$. 
Here $r$ is the number of rounds and $p$ is the number of pairs of pebbles
on board. Using (any) $p>n$ many pairs of pebbles avalable on the board \pa\ can win this game in $n+1$ many rounds.
In each round $0,1\ldots n$, \pe\ places a new pebble  on  a new element of $n+1$.
The edge relation in $n$ is irreflexive so to avoid losing
\pe\ must respond by placing the other  pebble of the pair on an unused element of $n$.
After $n$ rounds there will be no such element, so she loses in the next round.
 \pa\  lifts his \ws\ from the private \ef\ forth game ${\sf EF}_{n+1}^{n+1}(n+1, n)$ to the graph game on ${\bf At}_f=\At(\A_{n+1,n})$ 
\cite[pp. 841]{HH} forcing a
win using $n+3$ nodes. 
He bombards \pe\ with cones
having  common
base and distinct green  tints until \pe\ is forced to play an inconsistent red triangle (where indicies of reds do not match).
By Lemma \ref{n}, $\A_{n+, n}\notin
\bold S_c{\sf Nr}_n\CA_{n+3}$. Since $\A_{n+1, n}$
is finite, then $\A_{n+1, n}\notin \bold S\Nr_n\CA_{n+3}$, 
because $\A_{n+1}=\A_{n+1, n}$ (its canonical extension)  and $\D\in \bold S\Nr_n\CA_{n+3}\implies \A^+\in \bold S_c\Nr_n\CA_{n+3}$.
But $\A_{n+1,n}$ embeds into $\Cm\At\A$,
hence $\Cm\At\A$
is outside the variety $\bold S{\sf Nr}_n\CA_{n+3}$, as well. By Lemma \ref{n}, the required follows.

\end{proof}

On the other hand, the class $\bold S_c\Nr_n\CA_m$ is gripped.
For any $n<\omega$, the class ${\sf CRCA}_n$, and its elementary closure, namely,
the class of algebras satisfying the Lyndon conditions as defined in 
\cite{HHbook2}   is gripped.
The usual atomic game $G$ weakly grips, densely grips and 
grips ${\sf CRCA}_n$. 

\begin{theorem}\label{grip} The game $\bold H_{\omega}$ densely grips $\Nr_n\CA_{\omega}$. 
but  does not grip $\Nr_n\CA_{\omega}$.
\end{theorem}
\begin{proof} The first part follows from Theorem \ref{gripneat}. For the second more tricky part. Take  $\E\in {\sf Cs}_n$ to be the algebra in the first item of Theorem \ref{square}. 
We know that $\At\E\in \At\Nr_n\CA_{\omega}$ and $\Cm\At\E\in \Nr_n\CA_{\omega}$ 
but $\E\notin \Nr_n\CA_{n+1}(\supseteq \Nr_n\CA_{\omega}$). 
We show that \pe\ has a \ws\  in $\bold H_{\omega}(\At\E)$. 
 \pe's strategy dealing with $\lambda$--neat hypernetworks, where $\lambda$ is a constant label kept on short hyperedges is exactly like in the proof of Theorem \ref{gripneat}.
The rest of her \ws\
is to play $\lambda$--neat hypernetworks $(N^a, N^h)$ with $\nodes(N_a)\subseteq \omega$ such that
$(N^a)^+\neq 0$ (recall that $(N^a)^+$ is as defined in the proof of lemma \ref{n}).
In the initial round, let \pa\ play $a\in \At$.
\pe\ plays a network $N$ with $N^a(0, 1, \ldots n-1)=a$. Then $(N^a)^+=a\neq 0$.
The response to the cylindrifier move is exactly like in the first part of Lemma \ref{n} because $\E$ is completely representable
so $\E\in \bold S_c{\sf Nr}_n\CA_{\omega}$ \cite[Theorem 5.3.6]{Sayedneat}.
For transformation moves: if \pa\ plays
$(M, \theta),$ then it is easy to see that we have 
${(M^a\theta)}^+\neq 0$,
so this response is maintained in the next round.
For each $J\subseteq \omega$, $|J|=n$ say, for $\A\in \CA_{\omega}$, let $\Nr_J\A=\{x\in \A: {\sf c}_{l}x=x, \forall l\in \omega\sim  J\}.$ 
We have $\At\E\in \At{\sf Nr}_n\CA_{\omega}$, so assume that $\At\E=\At\mathfrak{Nr}_n\D$ with $\D\in \CA_{\omega}$. 
Then for all $y\in \Nr_J\D$,  where $J=\{i_0, i_1, \ldots, i_{n-1}\}$, $i_0, \ldots, i_{n-1}\in \omega$,  
the following holds for $a\in \At\E$ (*):  
${\sf s}_{i_0i_1\ldots i_{n-1}}a\cdot y\neq 0\implies  {\sf s}_{i_0 i_1\ldots i_{n-1}}a \leq y$.
Now we are ready to describle \pe's strategy in response to 
amalgamation moves. For better readability, we write $\bar{i}$ for $\{i_0, i_1, \ldots, i_{n-1}\}$, 
if it occurs as a set, and we write ${\sf s}_{\bar{i}}$ short for
${\sf s}_{i_0}{\sf s}_{i_1}\ldots {\sf s}_{i_{n-1}}$. 
Also we only deal with the network part of the game.
Now suppose that  \pa\ plays the amalgamation move $(M,N)$ where $\nodes(M)\cap \nodes(N)=\{\bar{i}\}$,
then $M(\bar{i})=N(\bar{i})$.  Let $\mu=\nodes(M)\sim \bar{i}$ and $v=\nodes(N)\sim \bar{i}.$
Then ${\sf c}_{(v)}M^+=M^+$ and
${\sf c}_{(u)}{N^+}={M}^+$.
Hence using (*), we have;
${\sf c}_{(u)}{M}^+={\sf s}_{\bar{i}}M(\bar{i})={\sf s}_{\bar{i}}N(\bar{i})={\sf c}_{(v)}{N}^+$
so ${\sf c}_{(v)}{M}^+={M}^+\leq {\sf c}_{(u)}{M}^+={\sf c}_{(v)}N^+$
and ${M}^+\cdot {N}^+=x\neq 0.$  So there is $L$ with $\nodes(L)=\nodes(M)\cup \nodes(N)$,
and ${L}^+\cdot x\neq 0$, 
thus ${L}^+\cdot {M}^+\neq 0$
and consequently ${L}\restr {\nodes(M)}={M}\restr {\nodes(M)}$,
hence $M\subseteq L$ and similarly $N\subseteq L$, so that $L$ is the required amalgam.\\
\end{proof}

\begin{example} Fix $2<n<\omega$. The game $\bold H$  
weakly and densely grips $\Nr_n\CA_{\omega}$ but $\bold H$ does not grip $\Nr_n\CA_{\omega}$. We devise an $\omega$--rounded {\it non-atomic} game 
$\bold G$ {\it gripping} $\Nr_n\CA_{\omega}$. By non-atomic, 
we mean that arbitray elements of the algebra not necessarily atoms are allowed during the play.
By the example in item 1 of Theorem \ref{square} ,  $\bold G$ is {\it strictly} stronger than $\bold H$. That is to say, 
\pe\ has a \ws\ in $\bold G\implies$ \pe\ has a \ws\ in $\bold H$, but the converse implication is false as, using the notation in {\it op.cit},  \pe\ 
has a \ws\ in $\bold H(\At\E)$ but does not have a \ws\ in $\bold G(\E)$.  

The game $\bold G$ is played on both $\lambda$--neat hypernetworks as defined for $\bold H$,
and complete labelled graphs (possibly by non--atoms)
with no consistency conditions.  The play at a certain point, like in $\bold H$,   will be a
$\lambda$--neat hypernetwork, call its
network part $X$, and we write
$X(\bar{x})$ for the atom the edge $\bar{x}$. By network part we mean forgetting $k$--hypedges getting 
non--atomic labels.
{\it An $n$-- matrix} is a finite complete graph with nodes including $0, \ldots, n-1$
with all edges labelled by {\it arbitrary elements} of $\B$. No consistency properties are assumed.
\pa\ can play an arbitrary $n$--matrix  $N$, \pe\ must replace $N(0, \ldots, n-1),$  by
some element $a\in \B$; this is a non-atomic move.
The final move is that \pa\ can pick a previously played $n$--matrix $N$, and pick any  tuple $\bar{x}=(x_0,\ldots, x_{n-1})$
whose atomic label is below $N(0, \ldots, n-1)$.
\pe\ must respond by extending  $X$ to $X'$ such that there is an embedding $\theta$ of $N$ into $X'$
 such that $\theta(0)=x_0\ldots , \theta(n-1)=x_{n-1}$ and for all $i_0, \ldots i_{n-1} \in N,$ we have
$$X(\theta(i_0)\ldots, \theta(i_{n-1}))\leq N(i_0,\ldots, i_{n-1}).$$
This ensures that in the limit, the constraints in
$N$ really define the element $a$.
Assume that $\B\in \CA_n$ is atomic and has countably many atoms. If \pe\ has a \ws\ in $\bold G(\B)$, 
then the extra move involving non--atoms labelling matrices, ensures that  that every $n$--dimensional element generated by
$\B$ in a  dilation $\D\in \RCA_\omega$ having base $M$, constructed from a \ws\ in $\bold G$ as the 
limit of the  $\lambda$--neat hypernetworks played during the game 
(and further assuming without loss that \pa\ plays every possible move)  
is already an element of $\B$.   For $k<\omega$, let $\bold G_k$ be the game $\bold G$ truncated to $k$ rounds, and let $\bold G^{ra}$ and $\bold G_k^{ra}$
be the relation algebra analogue of the game  obtained 
by adding the non--atomic move replacing $n$-matrices by $2-matrices$, to the game $H$ as defined for relation algebras \cite[Definition 28]{r}. 
\end{example}
Using the argument in the proof of the Theorem \ref{gripneat} replacing $\bold H$ by $\bold G$ we get: 
Assume that $2<n<m<\omega$. If there exists a countable atom structure $\alpha$ such that \pe\ has a \ws\ in $\bold G_k(\Cm\alpha)$
for all $k\in \omega$ and \pa\ has a \ws\ in $F^m$, then any class $\bold K$, such that ${\sf Nr}_n\CA_{\omega}\subseteq \bold K\subseteq \bold S_c{\sf Nr}_n\CA_m$,
is not elementary. We have already  proved the last result. The relation algebra case is more interesting. Undefined notation can be found in \cite{r}; detailed 
citation is given in the proof.
\begin{theorem} Assume that $2<m<\omega$. If there exists a countable atom structure $\alpha$ such that \pe\ has a \ws\ in $\bold G^{ra}_k(\Cm\alpha)$
for all $k\in \omega$ and \pa\ has a \ws\ in $F^m$, then any class $\bold K$, such that $\Ra\CA_{\omega}\subseteq \bold K\subseteq \bold S_c{\sf Ra}\CA_5$,
$\bold K$ is not elementary. 
\end{theorem}
\begin{proof} The analogous result can be obtained 
for relation algebras for $2<m<\omega$ obtained by replacing ${\sf Nr}_n$ by $\Ra$.
One uses the arguments in \cite[Theorem 39, 45]{r}, but  resorting to the game $\bold G^{ra}_k$ 
in place of $H_k$ $(k<\omega)$, as defined for relation algebras \cite[Definition 28]{r}. 
Now by assumption we have a countable relation algebra atom structure $\alpha$, 
for  which \pe\ has a \ws\ in $\bold G_k^{ra}(\Cm\alpha)$, for all $k<\omega$, and \pa\ has a \ws\ in $F^m(\alpha)$ with $F^m$ as defined  in \cite[Definition 28]{r}.
By the $\RA$ analogue of lemma \ref{n} proved in \cite[Theorem 33]{r}, we get that $\Cm\alpha\notin \bold S_c\Ra\CA_m$. 
The usual argument of taking an ultrapower of $\Cm\alpha$, followed by a downward  elementary chain argument, one gets 
a countable $\B\in \sf RA$, such that $\B\equiv \Cm\alpha$ and  
\pe\ has a \ws\ in $\bold G^{ra}(\B)$, 
so $\B\in \Ra\CA_{\omega}$ because $\bold G^{ra}$ grips $\sf Ra\CA_{\omega}$. Hence for any $\bold K$, such that $\sf Ra\CA_{\omega}\subseteq \bold K\subseteq \bold S_c\sf Ra\CA_5$, we have
$\Cm\alpha\notin \bold K$, $\B\in \bold K$ and $\Cm\alpha\equiv \B$.
\end{proof}

Fix $2<n<\omega$. Now we investigate the analogue of the result proved in Theorem \ref{rainbow} 
for several cylindric--like algebras, like Pinter's substitution algebras ($\Sc$) and 
quasi-polyadic (equality) algebras, denoted by $\sf QA$, ($\QEA$).  
We start with a lemma that generalizes \cite[Theorem 5.1.4]{Sayedneat} to  any class of algebras between $\Sc$ and $\QEA$.
Such a result was proved in several publictions of the author's dealing separately with different cases; relevant references can be found among the references 
in the article \cite{Sayedneat} collected at the end of \cite{1}. These references include (the non exhaustive list) \cite{Fm, note, MLQ, IGPL}.
Following \cite{HMT2} in the next proof ${\sf Pes}_n$ denotes the class of polyadic equality set algebras of dimension $n$. 
For a Boolean algebra with operators $\A\in {\sf BAO}$, $\Bl\A$ denotes its Boolean reduct.

\begin{lemma}\label{d}       
For any ordinal $\alpha>1$, and any uncountable cardinal $\kappa\geq |\alpha|$, there exist completely representable algebras $\A, \B\in \QEA_{\alpha}$, 
that are weak set algebras, such that $|\A|=|\B|=\kappa$, $\A\in {\sf Nr}_{\alpha}\QEA_{\alpha+\omega}$,  $\Rd_{sc}\B\notin {\sf Nr}_{\alpha}\Sc_{\alpha+1}$, 
$\A\equiv_{\infty, \omega}\B$ and $\At\A\equiv_{\omega, \infty}\At\B$.
\end{lemma}  
\begin{proof} Here we consider only finite dimensions which is all we need.
Also we count in dimension $2$. Fix $1<n<\omega$. Let $L$ be a signature consisting of the unary relation
symbols $P_0,P_1,\ldots, P_{n-1}$ and
uncountably many $n$--ary predicate symbols. $\Mo$ is as in \cite[Lemma 5.1.3]{Sayedneat}, but the tenary relations are replaced by 
$n$--ary ones, and we require that the interpretations of the $n$--ary relations in $\Mo$ 
are {\it pairwise disjoint} not only distinct. This can be fixed. 
In addition to pairwise
disjointness of $n$--ary relations, we require their symmetry, 
that is, permuting the variables does not change
their semantics. In fact the construction is presented this way in \cite{Fm, note}.
For $u\in {}^n n$, let $\chi_u$
be the formula $\bigwedge_{u\in {}^n n}  P_{u_i}(x_i)$. We assume that the $n$--ary relation symbols are indexed by (an uncountable set) $I$
and that  there is a binary operation $+$ on $I$, such that $(I, +)$ is an abelian group, and for distinct $i\neq j\in I$,
we have $R_i\circ R_j=R_{i+j}$. 
For $n\leq k\leq \omega$, let $\A_k=\{\phi^{\Mo}: \phi\in L_k\}(\subseteq \wp(^k\Mo))$,
where $\phi$ is taken in the signature $L$, and $\phi^{\Mo}=\{s\in {}^k\Mo: \Mo\models \phi[s]\}$. 

Let $\A=\A_n$, then $\A\in \sf Pes_n$ by the added symmetry condition.
Also $\A\cong \Nr_n\A_{\omega}$; the isomorphism is given by
$\phi^{\Mo}\mapsto \phi^{\Mo}$. The map is obviously an injective homomorphism; it is surjective, because $\Mo$ 
(as stipulated in \cite[ item (1) of lemma 5.1.3]{Sayedneat}), 
has quantifier elimination.
For $u\in {}^nn$, let $\A_u=\{x\in \A: x\leq \chi_u^{\Mo}\}.$ Then
$\A_u$ is an uncountable and atomic Boolean algebra (atomicity follows from the new disjointness condition)
and $\A_u\cong {\sf Cof}(|I|)$, the finite--cofinite Boolean algebra on $|I|$.
Define a map $f: \Bl\A\to \bold P_{u\in {}^nn}\A_u$, by
$f(a)=\langle a\cdot \chi_u\rangle_{u\in{}^nn+1}.$
Let $\P$ denote the
structure for the signature of Boolean algebras expanded
by constant symbols $1_u$, $u\in {}^nn$, ${\sf d}_{ij}$, and unary relation symbols
${\sf s}_{[i,j]}$ for each $i,j\in n$.
Then for each $i<j<n$, there are quantifier free formulas
$\eta_i(x,y)$ and $\eta_{ij}(x,y)$ such that
$\P\models \eta_i(f(a), b)\iff b=f({\sf c}_i^{\A}a),$
and $\P\models \eta_{ij}(f(a), b)\iff b=f({\sf s}_{[i,j]}a).$
The one corresponding to cylindrifiers is exactly like the $\CA$ case \cite[pp.113-114]{Sayedneat}.
For substitutions corresponding to
transpositions, it is simply $y={\sf s}_{[i,j]}x.$  The diagonal elements and the Boolean operations are easy to interpret. 
Hence, $\P$ is interpretable in $\A$, and the interpretation is one dimensional and
quantifier free. For $v\in {}^nn$, by the Tarski--Sk\"olem downward theorem, 
let  $\B_v$ be a countable elementary subalgebra of $\A_v$. (Here we are using the countable signature of $\PEA_n$).
Let $S_n (\subseteq {}^nn)$ be the set of permuations in $^nn$.

Take $u_1=(0, 1, 0, \ldots, 0)$ and $u_2=(1, 0, 0, \ldots, 0)\in {}^nn$. 
Let  $v=\tau(u_1,u_2)$  where $\tau(x,y)={\sf c}_1({\sf c}_0x\cdot {\sf s}_1^0{\sf c}_1y)\cdot {\sf c}_1x\cdot {\sf c}_0y$. We call $\tau$ an approximate
witness. It is not hard to show that $\tau(u_1, u_2)$ is actually the composition of $u_1$ and $u_2$, 
so that $\tau(u_1, u_2)$ is the constant zero map; which we denote by $\bold 0$; it is also in $^nn$. 
Clearly for every $i<j<n$,  ${\sf s}_{[i,j]} {}^{^nn}\{\bold 0\}=\bold 0\notin \{u_1, u_2\}$.
We can assume without loss that 
the Boolean reduct of $\A$ is the following product:
$$\A_{u_1}\times \A_{u_2}\times  \A_{\bold 0}\times\bold P_{u\in V\sim J} \A_u,$$
where $J=\{u_1, u_2, \bold 0\}$.
Let $$\B=((\A_{u_1}\times \A_{u_2}\times \B_{\bold 0}\times\bold P_{u\in V\sim J} \A_u), 1_u, {\sf d}_{ij}, {\sf s}_{[i,j]}x)_{i,j<n},$$
recall that  $\B_{\bold 0}\prec \A_{\bold 0}$ and $|\B_{\bold 0}|=\omega$,
inheriting the same interpretation.  Then by the Feferman--Vaught theorem,
we get that
$\B\equiv \A$.
Now assume for contradiction, that $\Rd_{sc}\B=\Nr_n\D,$ with $\D\in \Sc_{n+1}$.
Let $\tau_n(x,y)$, which we call an {\it $n$--witness}, be defined by ${\sf c}_n({\sf s}_n^1{\sf c}_nx\cdot {\sf s}_n^0{\sf c}_ny).$
By a straightforward, but possibly tedious computation, one can obtain  $\Sc_{n+1}\models \tau_n(x,y)\leq \tau(x,y)$ 
so that the approximate witness {\it dominates} the $n$--witness.
The term $\tau(x,y)$ does not use any spare dimensions, and it `approximates' the term $\tau_n(x,y)$ that
uses the spare dimension $n$. 
Let $\lambda=|I|$. For brevity, we write $\bold 1_{u}$ for $\chi_u^{\Mo}$. 
The algebra $\A$ can be viewed as splitting the atoms of the atom structure ${\bf At}=(^nn, \equiv_, \equiv_{ij}, D_{ij})_{i,j<n}$ each to uncountably many atoms.
We denote $\A$ by 
${\sf split}({\bf At}, \bold 1_{\bold 0}, \lambda).$ 
On the other hand, $\B$ can be viewed as splitting the same atom structure, each  atom -- except for one atom that is split into countably many atoms -- 
is also split into uncountably many atoms (the same as in $\A)$. 
We denote $\B$ by $\sf split({\bf At}, \bold 1_{\bold 0}, \omega).$ 

On the `global' level, namely, in the complex algebra of the finite
(splitted) atom structure $^nn$, these two terms are equal, the approximate witness
is the $n$--witness. The complex algebra $\Cm{}(^{n}n)$ does not `see' the $n$th dimension.
But in the algebras $\A$ and $\B$ (obtained after splitting),  the $n$--witness becomes then a {\it genuine witness},  not an approximate one.  The approximate witness 
{\it strictly dominates} the $n$--witness.  The $n$--witness using the spare dimension $n$, detects the cardinality twist that $L_{\infty, \omega}$, {\it a priori}, 
first order logic misses out on.
If the $n$--witness were term definable (in the case we have a full neat reduct of an algebra in only one extra dimension), then
it takes two uncountable component to an uncountable one,
and this is not possible for $\B$, because in $\B$, the target 
component  is forced to be 
countable. 
Now for $x\in \B_{u_1}$ and  $y\in \B_{u_2}$, we have
$$\tau_n^{\D}(x, y)\leq \tau_n^{\D}(\chi_{u_1}, \chi_{u_2})\leq \tau^{\D}(\chi_{u_1}, \chi_{u_2})=\chi_{\tau^{\wp(^nn)}}(u_1,u_2)=\chi_{\tau(u_1, u_2)}=\chi_{\bold 0}.$$
But for $i\neq j\in I$,
$\tau_n^{\D}(R_i^{\sf M}\cdot \chi_{u_1}, R_j^{\sf M}\cdot \chi_{u_2})=R_{i+j}^{\sf M}\cdot \chi_v$, and so    $\B_{\bold 0}$ will be uncountable,
which is impossible.
We now show that \pe\ has a \ws\ in an \ef\ back--and--forth game over the now atomic
$(\A, \B).$  At any stage of the game, if \pa\ places a pebble on one of
$\A$ or $\B$, \pe\ must place a matching pebble on the other
algebra.  Let $\b a = \la{a_0, a_1, \ldots, a_{m-1}}$ be the position
of the pebbles played so far (by either player) on $\A$ and let $\b b = \la{b_0, \ldots, b_{m-1}}$ be the the position of the pebbles played
on $\B$.  
 Denote $\chi_u^{\Mo}$, by $\bold 1_u$. Then \pe\ has to
maintain the following properties throughout the
game:
\begin{itemize}

\item for any atom $x$ (of either algebra) with $x\cdot \bold 1_{\bold 0}=0$, , then $x \in a_i$ iff $x\in b_i$,

\item $\b a$ induces a finite partition of $\bold 1_{\bold 0}$ in $\A$ of $2^m$
 (possibly empty) parts $p_i:i<2^m$ and the $\b b$
induces a partition of  $1_{u}$ in $\B$ of parts $q_i:i<2^m$ such that 
$p_i$ is finite iff $q_i$ is
finite and, in this case, $|p_i|=|q_i|$.
\end{itemize}

It is easy to see that \pe\ can maintain these two properties in every round. In this back--and--forth game, \pe\ will always find a matching pebble, 
because the pebbles in play are finite.  For each $w\in {}^nn$ 
the component $\B_w=\{x\in \B: x\leq \bold 1_v\}(\subseteq \A_w=\{x\in \A: x\leq \bold 1_v\}$) 
contains infinitely many atoms. 
For any $w\in V$, $|\At\A_w|=|I|$, while 
for  $u\in V\sim \{\bold 0\}$, $\At\A_u=\At\B_u$. For 
$|\At\B_{\bold 0}|=\omega$, but it is still an infinite set.
Therefore $\A\equiv_{\infty}\B$.
It is clear that the above argument
works for any $\C$ such that $\At\C=\At\B$, hence $\B\equiv_{\infty, \omega}\C.$

\end{proof}

To obtain the required result (generalizing Theorem \ref{iii} to any class of algebras betwen $\sf Sc$ and $\QEA$ (to be formulated in a while as Theorem \ref{rain}) 
the following changes are needed:
wn that \pe\ has a \ws\ in $\bold H_{\omega}^q(\At\C)$.

\begin{itemize}

\item modifying the game $\bold H$ to the $\QEA$ case, call the new game $\bold H^q$. Here one has to modify only the `network part' 
of $\lambda$--neat hypernetworks by adding a symmetry condition ($({\sf s}_{[i, j]}N(\bar{x})=N(\bar{x}\circ [i, j])$),
\item working with $\C_{\Z, \N}$ now expanded with the unary operations ${\sf s}_{[i,j]}$ $(i<j<n)$, call the resulting rainbow algebra $\C_{\Z, \N}^q(\in {\sf RQEA}_n)$,

\item generalizing \ws's of \pa\ given in Theorem \ref{iii} in $\bold H_\omega(\At\C_{\Z, \N})$ to the game $\bold H^q_k (\At\Rd_{sc}\C_{\Z, \N}^q)$ and 
\pa's \ws\ in $\bold H^q_{k}(\At\C_{\Z, \N}^q)$ for each $k<\omega$; the last is given in the proof of the same theorem. 
Then one uses the (easy) modification of lemma \ref{gripneat} to the $\QEA_n$ case 
(in the modification, for example in defining the weak model ${\cal M}_a$ $(a\in \alpha)$ one adds a clause to formulas for satisfiability of the unary connectives 
${\sf s}_{[i, j]}$ $(i<j<n)$ interpreted as swapping the $i$th and $j$th variables).
Using lemma \ref{n}, one  shows that $\Rd_{Sc}\C_{\Z, \N}^q\notin \bold S_c\Nr_n\Sc_{n+3}^{\sf ad}$. Exactly the $\CA$ case,
we get that any class $\sf K$ between  $\Sc$ nd $\QEA$, and any class between the two classes 
$\bold S_d\Nr_n\K_{\omega}\cap  {\sf CRK}_n$ and $\bold S_c\Nr_n\K_{n+3}^{\sf ad}$, 
is not elementary, 

\item to remove the $\bold S_d$, from $\bold S_d\Nr_n\K_{\omega}$ like was done in the $\CA$ case, 
one adjoins the (splitting) construction in 
\cite{note} for classes between $\Sc$ and $\QEA$ in place of the modification of the construction in \cite[Theorem 5.1.4]{Sayedneat}
(as appeared  \cite{IGPL}) addressing only the $\CA$ case. The construction in \cite{SL}
shows that this last item  is necessary. In {\it opcit}, an atomic $\C\in {\sf RQEA}_n$
such that $\Cm\At\C\in \Nr_n\QA_{\omega}$ and   $\At\C\in \At\Nr_n\QEA_{\omega}$, but $\Rd_{sc}\A\notin \Nr_n\Sc_{n+1}
(\supseteq \Nr_n\Sc_{\omega}$), 
is constructed. In fact, this $\C$ is simple (has no proper ideals), 
so that it is a set algebra of dimension $n$ 
which is the ${\sf QEA}_n$ generated by the same generators of
$\E$ in the first item of Theorem \ref{square} but now  in the full {\it quasipolyadic equality set algebra with top element $^n\mathbb{Q}$}.
Like in Theorem \ref{grip}, it can be shown that \pe\ has a \ws\ in $\bold H_{\omega}^q(\At\C)$.
\end{itemize}

So for diagonal free reducts of $\QEA$, namely, $\Sc$ and $\sf QA$, we obtain the weaker result which we formulate only for $\Sc$s.
The Theorem is true for any class between $\Sc$ and $\sf QA$. Recall that  ${\sf CRSc}_n$ denotes the class of completely representable
$\Sc_n$s.
\begin{theorem}\label{rain} Any class between ${\sf CRSc}_n\cap \Nr_n\Sc_{\omega}$ and $\bold S_c\Nr_n\Sc_{n+3}^{\sf ad}$ 
is not elementary. 
\end{theorem}
Complete additivity appears on the left hand side (giving a smaller class than 
the class $\bold S_c\Nr_n\Sc_{n+3}$)  due to the intrusion of lemma \ref{n}. 
A discrepancy that deserves to be highlighted here is that in the case of non--atom canonicity, 
proved in 
Theorem \ref{grip}, though the sameLlemma \ref{n} was also used, additivity did not interfere at all for diagonal free reducts (like $\Sc$s) 
because algebras, more specifically   dilations  of algebras used, were finite.

\begin{corollary} Let $2<n<\omega$, $\sf K$ any class between $\Sc$ and $\QEA$, $m>n$, and $k\geq 3$. 
Then the following classes are not elementary:
${\sf CRK}_n$ \cite{HH, MT}, $\Nr_n\K_{m}$,  $\Nr_2\K_{m}$, \cite{IGPL, Fm, note, MLQ} 
$\bold S_d\Nr_n\K_{n+k}$ and  $\bold S_c\Nr_n\K_{n+k}$.
\end{corollary} 
\begin{corollary} \label{Ra}Let $k\geq 5$. Then the classes 
$\sf CRRA$, $\Ra\CA_k$, $\bold S_d\Ra\CA_k$ and $\bold S_c\Ra\CA_k$ are not elementary \cite{HH, r}. 
The first two classes are not closed under
$\equiv_{\infty\omega}$, but are closed under ultraproducts. Furthermore, $\Ra\CA_\omega\subseteq \bold S_d\Ra\CA_{\omega}\subsetneq 
\bold S_c\Ra\CA_{\omega}$.  
\end{corollary}
\begin{proof} The first two classes are closed under ultraproducts because they are psuedo-elementary (reducts of elementary classes), 
cf. \cite[Item (2), p. 279]{HHbook}, 
\cite[Theorem 21]{r}. Proving the strictness of the last inclusion can be easily distilled from the proof of 
\cite[Theorem 36]{r}.
\end{proof}

\end{document}